\documentclass[11pt,a4paper]{article}
\usepackage[T1]{fontenc}
\usepackage{lmodern}
\usepackage[margin=27mm,headheight=14pt,headsep=7mm]{geometry}
\usepackage{amsmath,amssymb,amsthm,amscd,mathtools,mathrsfs}
\usepackage{enumitem,needspace,microtype,setspace,fancyhdr}
\usepackage[numbers,square,comma]{natbib}
\usepackage[hidelinks]{hyperref}
\allowdisplaybreaks[2]
\numberwithin{equation}{section}
\newtheorem{theorem}{Theorem}[section]
\newtheorem{proposition}[theorem]{Proposition}
\newtheorem{lemma}[theorem]{Lemma}
\newtheorem{corollary}[theorem]{Corollary}
\theoremstyle{definition}
\newtheorem{definition}[theorem]{Definition}
\theoremstyle{remark}
\newtheorem{remark}[theorem]{Remark}

\theoremstyle{plain}
\newtheorem*{theoremA}{Theorem A}
\newtheorem*{theoremB}{Theorem B}
\BeforeBeginEnvironment{proposition}{\Needspace{7\baselineskip}}
\BeforeBeginEnvironment{lemma}{\Needspace{6\baselineskip}}
\BeforeBeginEnvironment{corollary}{\Needspace{6\baselineskip}}
\BeforeBeginEnvironment{theorem}{\Needspace{8\baselineskip}}
\BeforeBeginEnvironment{theoremA}{\Needspace{8\baselineskip}}
\BeforeBeginEnvironment{theoremB}{\Needspace{8\baselineskip}}
\BeforeBeginEnvironment{definition}{\Needspace{5\baselineskip}}
\BeforeBeginEnvironment{proof}{\Needspace{3\baselineskip}}
\title{Weil Operators and Remainders of Anderson Generating Functions
over a Punctured Projective Line}
\author{Chuangqiang Hu$^1$, Xiao-Min Huang$^2$ and Guorui Ma$^{3,4}$\\[6pt]
\small $^1$School of Mathematics, Sun Yat-Sen University, Guangzhou 510275, P. R. China\\
\small $^2$School of Mathematics and Statistics, Guangdong University of Technology,\\
\small Guangzhou 510006, P. R. China\\
\small $^3$Shanghai Institute for Mathematics and Interdisciplinary Sciences, Shanghai, China\\
\small $^4$Research Institute of Intelligent Complex Systems, Fudan University, Shanghai, China\\[4pt]
\small E-mail: \href{mailto:huchq5@mail.sysu.edu.cn}{huchq5@mail.sysu.edu.cn} (C. Hu);\\
\small \href{mailto:mahuangxm@gdut.edu.cn}{mahuangxm@gdut.edu.cn} (X.-M. Huang);\\
\small \href{mailto:mgr18@tsinghua.org.cn}{mgr18@tsinghua.org.cn} (G. Ma, corresponding author)}
\date{}
\hypersetup{pdftitle={Weil Operators and Remainders of Anderson Generating Functions over a Punctured Projective Line},pdfauthor={Chuangqiang Hu, Xiao-Min Huang, Guorui Ma}}

\newcommand{\Fq}{\mathbb F_q}
\newcommand{\A}{A_\rho}

\newcommand{\Cinf}{\mathbb C_\infty}
\newcommand{\KPhi}{\mathcal K_\Phi}

\newcommand{\OO}{\mathcal O}
\newcommand{\WW}{\mathcal W}
\newcommand{\DD}{\mathfrak D}
\newcommand{\EE}{\mathcal E}

\newcommand{\Alt}{\operatorname{Alt}}
\newcommand{\Frac}{\operatorname{Frac}}
\newcommand{\Tr}{\operatorname{Tr}}

\newcommand{\Hom}{\operatorname{Hom}}
\newcommand{\Spec}{\operatorname{Spec}}
\newcommand{\Span}{\operatorname{Span}}

\newcommand{\wt}{\widetilde}

\newcommand{\tensor}{\otimes}

\newcommand{\dfrak}{\mathfrak d}

\newcommand{\pideal}{\mathfrak p}
\newcommand{\aideal}{\mathfrak a}

\newcommand{\End}{\operatorname{End}}

\begin{document}

\maketitle
\begin{abstract}
Let $A_\rho$ be the coordinate ring of the projective line over
$\Fq$ with a closed point of degree $N$ removed. We give finite formulas
for higher Weil operators in a basis of Riemann--Roch spaces, which recover
the formulas of Hu--Ou--Yang when $N=1$. We identify
these operators with representatives of the residue Casimir tensors and
obtain a determinant formula for the Weil pairing, relative to a chosen
determinant realization. For a normalized rank-one Drinfeld module, we
compute the remainder of its Anderson generating function modulo any
nonconstant $p\in A_\rho$. The formula includes multiple zeros and a zero
at the ordinary point at infinity. Its last coefficient is a nonzero
scalar multiple of an exponential division value. When $(p)$ is a finite
principal prime, a value of the associated special function at one zero
of $p$ generates the torsion field after adjoining that zero. We give an
explicit $q$-linear formula recovering a torsion generator from this
value. We also specialize the evaluation theorem of Gazda--Maurischat
and compute the different as the principal ideal $(\partial_xp)$.
\end{abstract}


\medskip

\section{Introduction}\label{sec:introduction}

Drinfeld modules are function-field analogues of elliptic curves
\cite{Drinfeld1974,Goss1996,Papikian2023}. The Carlitz module and Hayes's
explicit class field theory provide basic examples over rational function
fields \cite{Carlitz1935,Hayes1979}. Anderson generating functions and
special functions also occur in formulas for zeta values, $L$-values and
Gauss--Thakur sums
\cite{AndersonThakur1990,Pellarin2012,AnglesPellarin2015}.

Let $\rho\in\Fq[t]$ be monic irreducible of degree $N$, and let
$P_\rho$ be the corresponding closed point of $\mathbf P^1_{\Fq}$. We
study explicit formulas over
\[
 \A=H^0(\mathbf P^1_{\Fq}\setminus\{P_\rho\},\mathcal O).
\]
For $N=1$, this is a polynomial ring. For $N>1$, extending the constant
field to $H=\Fq(\zeta)$, where $\rho(\zeta)=0$, gives the localization
\eqref{eq:A_H-local} of $H[x]$, with $x=(t-\zeta)^{-1}$.
Our purpose is to compute higher Weil operators and remainders of
Anderson generating functions in explicit Riemann--Roch bases of this
ring.

\subsection{The pairing and the dual basis tensor}

Van der Heiden constructed a determinant Weil pairing for Drinfeld
modules over a general curve \cite{vanderHeiden2004}. Hu--Huang expressed
the pairing through a basis and its residue dual
\cite[Theorem~5.9 and Definition~5.10]{HuHuang2024}.
For a polynomial ring, Hu--Ou-Yang identified the higher Weil operator
with the remainder of a product of rank-two operators
\cite[Definition~2.4, Corollary~2.22 and Theorem~2.23]{HuOuYang2026}.
Hu--Yu related finite residue tensors to geometric kernels
\cite[Theorems~6.2 and~6.3]{HuYu2026}.

We give finite formulas in the basis $T_k$ of
Proposition~\ref{prop:factorial-basis-imrn}.
Theorem~\ref{thm:universal-O-imrn} proves regularity and the determinant
identity. Proposition~\ref{thm:O-combinatorial-revised} computes the
entries of that determinant. Proposition~\ref{thm:pairing-equals-O-imrn}
identifies the resulting function with the representative of the
residue tensor in the specified Riemann--Roch space.

We state these results after fixing the notation. Let
$\zeta$ be a fixed root of $\rho$ and put $H=\Fq(\zeta)$. For Drinfeld
modules over $\A$ the infinite place is $P_\rho$. If $p\in\A$ has exact
pole order $d$ at $P_\rho$, put $n=Nd$ and
$V_p=\Span_H\{e_0,\ldots,e_{n-1}\}$, where $e_0=1$ and
$e_i(t)=\prod_{\ell=1}^{i}(t-\zeta^{q^{-\ell}})^{-1}$ for $1\le i<n$.
Write $A_H=\A\otimes_{\Fq}H$ and $x(t)=(t-\zeta)^{-1}$. For functions
$F_0,\ldots,F_{r-1}$, the alternant is
$\Alt(F_0,\ldots,F_{r-1})=\det(F_{i-1}(t_j))_{i,j=1}^r$. For $p\in A_H$
and $r\ge2$, the higher Weil operator is the alternant quotient
\begin{equation}\label{eq:intro-O-revised}
\OO_{p,\zeta}^{[r]}=
\frac{\Alt(1,p,xp,\ldots,x^{r-2}p)}
{\prod_{i<j}(x_j-x_i)},\qquad x_i=x(t_i),
\end{equation}
a priori a rational function in $t_1,\ldots,t_r$ over
$H$. For every nonzero $p\in A_H$, the divided difference kernel $K_p$
and the differential $\omega_p$ in
\begin{equation*}
 K_p(u,v)=\frac{p(u)-p(v)}{x(u)-x(v)},\qquad \omega_p=-\frac{dx}{p}
\end{equation*}
link the operator to the residue pairing.

\begin{theoremA}[Explicit higher Weil operator]
For $p\in A_H$ and $r\ge2$, the tensor \eqref{eq:intro-O-revised} is
symmetric and belongs to the $r$-fold tensor power
$A_H^{\otimes_H r}$. It has the sum formula and the determinant
formula in the divided differences of \eqref{eq:dd-rec-imrn}
\begin{align*}
 \OO_{p,\zeta}^{[r]}
 &=\sum_{\nu=1}^r\prod_{i\ne\nu}\frac{p(t_i)}{x_i-x_\nu}\\
 &=\det\bigl((x^{r-2+j-i}p)[x_1,\ldots,x_r]\bigr)_{i,j=1}^{r-1}.
\end{align*}
For $p\in\A\setminus\Fq$ of exact pole order $d$, the class of
$K_p$ in $(A_H/(p))^{\otimes_H 2}$ is the residue Casimir tensor of
$A_H/(p)$, and $\OO_{p,\zeta}^{[r]}$ is the representative of the
higher Casimir tensor
\begin{align*}
 \OO_{p,\zeta}^{[r]}
 =\operatorname{rem}_{p,t_r}\left(\prod_{j<r}K_p(t_j,t_r)\right)
\end{align*}
where $\operatorname{rem}_{p,t_r}$ takes the class modulo $p(t_r)$
to its unique representative in $V_p$ in the last variable.
\end{theoremA}

Expanding every entry of the determinant formula in the basis $(T_k)$ gives the finite sums of
Proposition~\ref{thm:O-combinatorial-revised}. Their denominators are
units of $A_H^{\otimes_H r}$, not diagonal factors $x_i-x_j$. The class
of $\OO_{p,\zeta}^{[r]}$ in $(A_H/(p))^{\otimes_H r}$ is the higher
dual basis tensor. Under the hypotheses of
Proposition~\ref{thm:actual-Weil-formula}, including that $(p)$ is a
principal prime away from the characteristic and a compatible
determinant realization is fixed, it gives the scalar coordinate of the Weil pairing
with respect to $\omega_p$. The regularity and determinant
identities are proved in Theorem~\ref{thm:universal-O-imrn}, and the
representative assertion is Proposition~\ref{thm:pairing-equals-O-imrn}.

\subsection{Generating functions and remainder results}

We use the normalized rank-one Drinfeld module $\Phi$ of
Hu--Huang--Yau \cite[Definition~6.6]{HuHuangYau2026}.
Its exponential, period and generating function are recalled in
Section~\ref{sec:drinfeld}. The exact remainder formula
\eqref{eq:exact-analytic-division} follows from the rational identity
\eqref{eq:exact-kernel-division}. It applies to the entire zero divisor
of $p$, with multiplicities. At a finite principal prime,
Proposition~\ref{thm:one-value-torsion} gives a $q$-linear formula
recovering a torsion generator from one special value.

We keep the notation of Section~1.1 and add the analytic
notation. Let $\theta$ be the generic characteristic image of $t$, write
$K=\Fq(t)$, and let $\Cinf$ be the completion of an algebraic closure
of $K$ at $P_\rho$. Put
\begin{align*}
 \beta(u)=\frac{u-\zeta^{q}}{u-\zeta},\qquad
 \gamma_\zeta=\beta(\theta)^{-1}.
\end{align*}
The Anderson generating function is built from the normalized module
$\Phi$ introduced in Section~\ref{sec:drinfeld} and the exponential
evaluation map $\EE_{\varphi,z}$ of
Definition~\ref{def:exponential-action-revised}. For $z\in\Cinf$ it is
\begin{equation*}
 H_\Phi(z;t)=\gamma_\zeta\,
 \EE_{\Phi,z}\left(\frac{\beta(u)}{x(u)-x(t)}\right),
\end{equation*}
where the kernel is a rational function of the independent variable $u$,
and its evaluation is the convergent series defined and justified in
Section~\ref{sec:drinfeld}. Next, fix
$p\in\A\setminus\Fq$ of exact pole order $d$ at $P_\rho$, with leading
coefficient $\alpha_n\ne0$, and let $\mathscr A_p$ be the product of the
analytic algebras on sufficiently small discs around all zeros of $p$,
including $P_\infty$ when it is a zero. Modulo $p\mathscr A_p$, every
$F\in\mathscr A_p$ has a unique representative
\begin{equation*}
 [F]_p=\sum_{i=0}^{n-1}C_i(F)e_i(t)\in\Cinf\otimes_H V_p,
\end{equation*}
the representative space of Section~\ref{sec:pairing}.
With $K_p$ the divided difference kernel of Theorem~A, define the
remainder
\begin{equation*}
 R_p(z;t)=\gamma_\zeta\,
 \EE_{\Phi,z}\left(\frac{\beta(u)}{p(u)}K_p(u,t)\right).
\end{equation*}
Finally, let $\iota:\A\to\Cinf$ be the embedding $a\mapsto a(\theta)$,
let $\wt\pi_\Phi$ be the chosen period with
$\ker\exp_\Phi=\wt\pi_\Phi\iota(\A)$, and put
$\omega_\Phi=H_\Phi(\wt\pi_\Phi;\cdot)$. For the field statement, let
$\pideal=(p)$ be a finite principal prime, let $\lambda$ be a finite
zero of $p$, and let $\KPhi\subset\Cinf$ be a field containing $\iota(K)$
and the coefficients of $\Phi$. Theorem~B uses the normalization of $\Phi$
fixed above. Its field statement is restricted to finite principal
primes and keeps $\lambda$ in the base field.

\begin{theoremB}[Exact remainder and torsion fields]
For every $z\in\Cinf$, the difference $H_\Phi(z;t)-R_p(z;t)$ is
divisible by $p(t)$ in $\mathscr A_p$, and $R_p$ is the $p$-remainder
of $H_\Phi$:
\begin{align*}
 H_\Phi(z;t)-R_p(z;t)
 &=p(t)H_\Phi\left(\frac{z}{p(\theta)};t\right),\qquad
 [H_\Phi(z;t)]_p=R_p(z;t).
\end{align*}
The last coefficient of the remainder in the basis $e_i$ is
\begin{align*}
 C_{n-1}(H_\Phi(z;\cdot))
 =\alpha_n\gamma_\zeta\exp_\Phi\left(\frac{z}{p(\theta)}\right).
\end{align*}
A single value of $\omega_\Phi$ generates the full $\pideal$-torsion
extension
\begin{align*}
 \KPhi(\lambda,\omega_\Phi(\lambda))
 =\KPhi(\lambda,C_{n-1}(\omega_\Phi))
 =\KPhi(\lambda,\Phi[\pideal]).
\end{align*}
The last coefficient is recovered from $\omega_\Phi(\lambda)$ by the
explicit $q$-linear polynomial in \eqref{eq:q-linear-leading-recovery}.
\end{theoremB}

The exact identity is proved in
Theorem~\ref{thm:H-remainder-imrn}. The field statement is its
consequence in Proposition~\ref{thm:one-value-torsion}.

As a further application, in Section~\ref{sec:gauss} we specialize the
evaluation formalism of Gazda--Maurischat \cite[Theorem~3.11 and
Theorem~5.2]{GazdaMaurischat2021} to the chosen branch. We apply their
isomorphism \eqref{eq:delta-p-GM} and evaluation map to the special
function $\omega_\Phi$ of $\Psi$, the rank one module fixed in
Section~\ref{sec:drinfeld}. We also compute the different as the
principal ideal $(\partial_xp)$ with
$\partial_xp=-(t-\zeta)^2\,dp/dt$. This yields the
explicit Gauss--Thakur evaluation of
Proposition~\ref{prop:GM-specialized}. The special function module is
$\Psi$, while the torsion extension in Theorem~B belongs to $\Phi$.

\subsection{Scope and organization}\label{sec:conclusion}
The new content over \cite{HuOuYang2026} is precisely the case $N>1$,
in which the coordinate ring is no longer the polynomial ring. The
contribution is the explicit calculation over $A_\rho$ in the
bases $T_k$ and $e_i$, including the representative identity
\eqref{eq:W-equals-O-imrn} and the recovery formula
\eqref{eq:explicit-torsion-recovery}. The normalized module and its
period come from \cite{HuHuangYau2026}. The general determinant pairing
comes from \cite{vanderHeiden2004}, and the Gauss--Thakur evaluation is
an application of \cite[Theorem~5.2]{GazdaMaurischat2021}.

Section~\ref{sec:factorialandoperators} proves the divided-difference
formulas. Section~\ref{sec:motive-Weil} gives the residue and motive
construction used to identify the pairing. Section~\ref{sec:pairing}
computes its coordinates on the punctured line.
Section~\ref{sec:drinfeld} proves the remainder formula and the torsion
field statement. Section~\ref{sec:gauss} computes the different and
specializes the Gauss--Thakur evaluation theorem.

\section{Divided differences and higher Weil operators}\label{sec:factorialandoperators}

We first describe $A_H$ and the basis $T_k$. We then prove the
divided-difference identities needed for the higher Weil operators.

\subsection{The scalar extension and Riemann--Roch bases}

 Let $\rho\in\Fq[t]$ be monic irreducible of degree $N$, let $P_\rho$ be its closed point on $\mathbf P^1_{\Fq}$, and let $\zeta$ be a fixed root of $\rho$. Put
\begin{align*}
 H=\Fq(\zeta)=\mathbb F_{q^N},
 \qquad
 \A=H^0\bigl(\mathbf P^1_{\Fq}\setminus\{P_\rho\},\mathcal O\bigr),
 \qquad
 A_H=\A\otimes_{\Fq}H.
\end{align*}
Write $\zeta^{(j)}=\zeta^{q^j}$, reading negative exponents
through the inverse Frobenius on $H$, and write $x(t)=(t-\zeta)^{-1}$.
Put $c_m=\zeta-\zeta^{(-m)}$. The sequence $(c_m)$ is $N$-periodic, and
$c_m=0$ if and only if $N\mid m$. Over $H$, the point $P_\rho$ splits
into $N$ points. Write $P_s$ for the point
$t=\zeta^{(-s)}$, $0\le s<N$. Hence
\begin{equation}\label{eq:A_H-local}
 A_H=H\left[\frac1{t-\zeta},\frac1{t-\zeta^{(-1)}},\ldots,\frac1{t-\zeta^{(-(N-1))}}\right]
 =H\left[x,(1+c_1x)^{-1},\ldots,(1+c_{N-1}x)^{-1}\right].
\end{equation}
For $r\ge1$, put $A_{H,r}=A_H^{\tensor_H r}$ and write $x_i=x(t_i)$. The ring $A_{H,r}$ is a localization of $H[x_1,\ldots,x_r]$. It is therefore a domain and a unique factorization domain.

For $s\in\mathbb Z$ and $k\ge0$, the basis functions are
\begin{align*}
 T_k^{(s)}(t)=\prod_{\ell=0}^{k-1}(t-\zeta^{(s-\ell)})^{-1},\qquad
 T_k=T_k^{(0)},
\end{align*}
where an empty product is $1$. In particular $T_1=x$. Cancellation of factors gives
\begin{equation}\label{eq:shift-identity-imrn}
 T_j^{(s-k)}=\frac{T_{j+k}^{(s)}}{T_k^{(s)}},
 \qquad
 T_j^{(-1)}=\frac{T_{j+1}}{T_1}.
\end{equation}
The successive factors of $T_k$ follow the periodic pole sequence $\zeta,\zeta^{(-1)},\zeta^{(-2)},\ldots$. The same products occur in \cite[Lemma~6.22 and Corollary~6.24]{HuHuangYau2026}. Proposition~\ref{prop:factorial-basis-imrn} records their Riemann--Roch basis property.

\begin{proposition}\label{prop:factorial-basis-imrn}
For \(s\in\mathbb Z\) and \(n\ge0\), put
\(D_n^{(s)}=\sum_{\ell=0}^{n-1}[\zeta^{(s-\ell)}]\).
Then $T_0^{(s)},\ldots,T_n^{(s)}$ form an $H$-basis of the
Riemann--Roch space $L(D_n^{(s)})$. Taking $s=0$, every $p\in A_H$ has a unique finite expansion
$p=\sum_{k=0}^{n}\alpha_kT_k, \quad \alpha_k\in H$.
If $p\in\A$ has pole order at most $d$ at $P_\rho$, one may take $n=Nd$.
\end{proposition}

\begin{proof}
Since \(\mathbf P^1_H\) has genus \(0\) and \(\deg D_n^{(s)}=n\), Riemann--Roch gives
\(\dim_H L(D_n^{(s)})=n+1\). For \(k\ge1\), the quotient
$\frac{T_k^{(s)}}{T_{k-1}^{(s)}}=\frac{1}{t-\zeta^{(s-k+1)}}$
has one additional simple pole at the new summand of \(D_k^{(s)}-D_{k-1}^{(s)}\). Hence
\(T_k^{(s)}\in L(D_k^{(s)})\setminus L(D_{k-1}^{(s)})\). Induction on \(k\) shows that
\(T_0^{(s)},\ldots,T_n^{(s)}\) are linearly independent. Their number equals
\(\dim_H L(D_n^{(s)})\), so they form a basis. If \(p\in\A\) has pole order at most \(d\) at \(P_\rho\), then after base change its pole divisor is bounded by \(d(P_0+\cdots+P_{N-1})=D_{Nd}^{(0)}\). Indeed, the first \(Nd\) terms of the periodic list \(\zeta,\zeta^{(-1)},\zeta^{(-2)},\ldots\) contain every split point exactly \(d\) times.
\end{proof}

\subsection{Divided differences and the action of \texorpdfstring{$x$}{x}}

This subsection records the divided difference identities in \(x=(t-\zeta)^{-1}\)
and proves that the apparent poles along the diagonal cancel in $A_{H,r}$.

For \(F\in A_H\), set \(F_i=F(t_i)\). Define the divided difference by
\begin{align}
 F[x_1]&=F_1,\notag\\
 F[x_1,\ldots,x_r]
 &=\frac{F[x_1,\ldots,x_{r-1}]-F[x_2,\ldots,x_r]}{x_1-x_r}
 \qquad(r\ge2).\label{eq:dd-rec-imrn}
\end{align}
We first record the divided difference formulas cited in
the operator calculations.

\begin{lemma}\label{lem:dd-basic-imrn}
For \(r\ge1\), the divided difference has the barycentric form
\begin{equation}\label{eq:dd-bary-imrn}
 F[x_1,\ldots,x_r]
 =\sum_{\nu=1}^{r}
 \frac{F_\nu}
 {\prod_{\substack{1\le i\le r\\ i\ne \nu}}(x_\nu-x_i)}.
\end{equation}
If \(F,G\in\Frac(A_H)\), the product formula is
\begin{equation}\label{eq:dd-Leibniz-imrn}
 (FG)[x_1,\ldots,x_r]
 =\sum_{\ell=1}^{r}
 F[x_1,\ldots,x_\ell]G[x_\ell,\ldots,x_r].
\end{equation}
More generally, let \(F_0,\ldots,F_{k-1}\) be functions,
 and put
\(s_\ell=m_0+\cdots+m_{\ell-1}\), \(s_0=0\).
Then the product of \(k\) functions satisfies
\begin{equation}\label{eq:dd-multi-Leibniz-imrn}
 \left(\prod_{\ell=0}^{k-1}F_\ell\right)[x_1,\ldots,x_r]
 =
 \sum_{\substack{m_0,\ldots,m_{k-1}\ge0\\
 m_0+\cdots+m_{k-1}=r-1}}
 \prod_{\ell=0}^{k-1}
 F_\ell[x_{s_\ell+1},\ldots,x_{s_\ell+m_\ell+1}].
\end{equation}
\end{lemma}

\begin{proof}
The barycentric identity \eqref{eq:dd-bary-imrn} follows by induction from the recursion \eqref{eq:dd-rec-imrn}. Applying the same recursion to \(FG\), followed by the induction hypothesis, gives \eqref{eq:dd-Leibniz-imrn}. The interior terms cancel pairwise. Repeated use of \eqref{eq:dd-Leibniz-imrn} gives the multiple product formula \eqref{eq:dd-multi-Leibniz-imrn}. The blocks of nodes overlap at their endpoints, which accounts for the indices \(s_\ell\).
\end{proof}

Formula \eqref{eq:dd-bary-imrn} contains apparent diagonal poles.
The next lemma shows that these poles cancel for every function in
\(A_H\).

\begin{lemma}\label{lem:inverse-linear-dd-imrn}
For \(c\in H\) and \(r\ge1\), the divided difference of an inverse linear function is
\begin{equation}\label{eq:inverse-linear-dd-imrn}
 \bigl((1+cx)^{-1}\bigr)[x_1,\ldots,x_r]
 =\frac{(-c)^{r-1}}{\prod_{i=1}^{r}(1+cx_i)}.
\end{equation}
Consequently, for $F\in A_H$ one has
\begin{equation}\label{eq:dd-regular-imrn}
 F[x_1,\ldots,x_r]\in A_{H,r}
 \qquad(F\in A_H).
\end{equation}
\end{lemma}

\begin{proof}
\eqref{eq:inverse-linear-dd-imrn} follows by induction from \eqref{eq:dd-rec-imrn}. The presentation of \(A_H\) in \eqref{eq:A_H-local} shows that \(A_H\) is generated by \(x\) and the inverses \((1+c_\ell x)^{-1}\), \(1\le \ell<N\). Divided differences of powers of \(x\) are polynomials. The inverse linear formula \eqref{eq:inverse-linear-dd-imrn}, together with the multiple product rule \eqref{eq:dd-multi-Leibniz-imrn}, proves \eqref{eq:dd-regular-imrn}.
\end{proof}

To make the cancellation of diagonal poles explicit, note the factorization
\( T_k(t)=\prod_{\ell=0}^{k-1}\frac{x(t)}{1+c_\ell x(t)}.
\)
The divided differences of these factors have the closed form
\begin{equation}\label{eq:B-ad-imrn}
 B_{\ell,m}(y_0,\ldots,y_m)=
 \frac{X}{(1+c_\ell X)}[y_0,\ldots,y_m]=
 \begin{cases}
 \displaystyle\frac{y_0}{1+c_\ell y_0},&m=0,\\[2mm]
 \displaystyle\frac{(-c_\ell)^{m-1}}
 {\prod_{j=0}^{m}(1+c_\ell y_j)},&m\ge1.
 \end{cases}
\end{equation}
When \(c_\ell=0\), the second line of \eqref{eq:B-ad-imrn} equals \(1\) for \(m=1\) and \(0\) for \(m\ge2\).

\begin{proposition}[Divided differences of $T_k$]\label{prop:factorial-dd-imrn}
For \(k\ge1\) and \(r\ge1\), the divided difference of \(T_k\) has the finite expansion
\begin{equation}\label{eq:factorial-dd-imrn}
 T_k[x_1,\ldots,x_r]
 =
 \sum_{\substack{m_0,\ldots,m_{k-1}\ge0\\
 m_0+\cdots+m_{k-1}=r-1}}
 \prod_{\ell=0}^{k-1}
 B_{\ell,m_\ell}
 (x_{s_\ell+1},\ldots,x_{s_\ell+m_\ell+1}),
\end{equation}
with the same $s_\ell$ as above.
In particular, the only denominators on the right of
\eqref{eq:factorial-dd-imrn} are products of the units
\(1+c_\ell x_i\).
\end{proposition}

\begin{proof}
For \(m=0\), the divided difference of \(X/(1+c_\ell X)\) is its value at one node. For \(m\ge1\) and \(c_\ell\ne0\), write
$\frac{X}{1+c_\ell X}=c_\ell^{-1}\bigl(1-(1+c_\ell X)^{-1}\bigr)$
and apply \eqref{eq:inverse-linear-dd-imrn}. If \(c_\ell=0\), the factor is \(X\), its first divided difference is \(1\), and all higher divided differences vanish. Formula \eqref{eq:factorial-dd-imrn} now follows from the multiple product identity \eqref{eq:dd-multi-Leibniz-imrn}.
\end{proof}

We next give a recurrence and a chain formula for \(x^sp\) in the coefficients relative to $(T_k)$.

\begin{proposition}\label{prop:x-multiplication-imrn}
Suppose $p=\sum_{j\ge0}\alpha_jT_j$ is a finite sum and write
$x p=\sum_{k\ge0}\beta_kT_k$. Then
\begin{equation}\label{eq:beta-rec-imrn}
\beta_0=0,\quad \beta_1=\alpha_0,\quad
 \beta_{j+1}=\alpha_j+c_{j-1}\beta_j\quad(j\ge1).
\end{equation}
Equivalently,
\begin{equation}\label{eq:beta-closed-imrn}
 \beta_{j+1}=\sum_{\ell=0}^{j}\alpha_\ell\prod_{u=\ell}^{j-1}c_u.
\end{equation}
The resulting expansion is finite.
\end{proposition}

\begin{proof}
Since $x$ and every $T_k$ with $k>0$ vanish at the ordinary point $P_\infty=[1:0]$, evaluation there first gives $\beta_0=0$. The identity
$(t-\zeta)T_k=T_{k-1}-c_{k-1}T_k$
follows from $T_{k-1}=(t-\zeta^{(-(k-1))})T_k$. Multiply the expansion of $xp$ by $t-\zeta$ and compare the coefficient of $T_j$. This gives \eqref{eq:beta-rec-imrn}. Iteration gives
\eqref{eq:beta-closed-imrn}. To prove finiteness, suppose
\(\alpha_\ell=0\) for \(\ell>n\). If \(j\ge n+N\), then every interval
\(\{\ell,\ell+1,\ldots,j-1\}\), with \(0\le\ell\le n\), contains an integer
divisible by \(N\). The corresponding product in
\eqref{eq:beta-closed-imrn} therefore contains a factor \(c_{mN}=0\).
Thus \(\beta_{j+1}=0\) for all \(j\ge n+N\).
\end{proof}

For an explicit iteration, define the chain coefficients
\begin{align}
 C_0(k,m)&=\delta_{k,m},\notag\\
 C_s(k,m)&=
 \sum_{m=i_0<i_1<\cdots<i_s=k}
 \prod_{\ell=1}^{s}
 \prod_{v=i_{\ell-1}}^{i_\ell-2}c_v
 \qquad(s\ge1).\label{eq:Cpath-imrn}
\end{align}
The sum is zero if no such chain exists. Iterating
Proposition~\ref{prop:x-multiplication-imrn} yields the finite chain
expansion required in the determinant entries.

\begin{corollary}\label{cor:iterated-x-imrn}
For $s,m\ge0$ and a finite sum $p=\sum_m\alpha_mT_m$,
\begin{align*}
 x^sT_m&=\sum_{k\ge0}C_s(k,m)T_k,\\
 x^sp&=\sum_{k\ge0}\left(\sum_m C_s(k,m)\alpha_m\right)T_k.
\end{align*}
All sums are finite.
\end{corollary}

\begin{proof}
The case $s=1$ is Proposition~\ref{prop:x-multiplication-imrn} applied to one basis element. Iteration records a strictly increasing sequence of basis indices. Multiplying the one-step coefficients along that sequence gives \eqref{eq:Cpath-imrn}.
\end{proof}

\subsection{Alternants and universal formulas}

We begin with the alternant definition and compare it with barycentric,
determinantal, and recursive descriptions.

For $F_0,\ldots,F_{r-1}\in A_H$, write
$\Alt(F_0,\ldots,F_{r-1})=\det(F_{i-1}(t_j))_{i,j=1}^r$
and $V_x=\prod_{i<j}(x_j-x_i)$.

The following definition, given over the fraction field and fixing normalization
for all ranks, is shown to be regular in Theorem~\ref{thm:universal-O-imrn}.

\begin{definition}\label{def:split-Weil-imrn}
Set $\OO^{[1]}_{p,\zeta}=1$. For $r\ge2$, the higher Weil operator is the alternant quotient
\begin{equation}\label{eq:O-alt-imrn}
 \OO^{[r]}_{p,\zeta}(t_1,\ldots,t_r)
 =\frac{\Alt(1,p,xp,\ldots,x^{r-2}p)}{V_x}
 \in\Frac(A_{H,r}).
\end{equation}
When $\zeta$ is fixed, we write $\OO_p^{[r]}$.
\end{definition}

For every $s\ge0$, the divided difference
$(x^sp)[x_1,\ldots,x_r]$ is regular by
Lemma~\ref{lem:inverse-linear-dd-imrn}.
We now compare the alternant, barycentric, determinant, and recursive descriptions.

The following theorem gives the identities used in Theorem~A. Its polynomial case is \cite[Corollary~2.22 and Theorem~2.23]{HuOuYang2026}. The proof below applies directly to rational $p\in A_H$ and proves regularity in $A_{H,r}$.

\begin{theorem}\label{thm:universal-O-imrn}
For $r\ge2$, the operator $\OO_p^{[r]}$ lies in $A_{H,r}$ and is symmetric. It has the following three forms.
\begin{align}
 \OO_p^{[r]}
 &=\sum_{\nu=1}^{r}\prod_{\substack{1\le i\le r\\ i\ne \nu}}\frac{p(t_i)}{x_i-x_\nu}\label{eq:O-bary-imrn}\\
 &=\det\left((x^{r-2+j-i}p)[x_1,\ldots,x_r]\right)_{1\le i,j\le r-1}\label{eq:O-Toeplitz-imrn}\\
 &=\frac{p(t_1)\OO_p^{[r-1]}(t_2,\ldots,t_r)
 -p(t_2)\OO_p^{[r-1]}(t_1,t_3,\ldots,t_r)}{x_1-x_2}.
 \label{eq:O-rec-imrn}
\end{align}
The recursion starts with
$\OO_p^{[2]}(t_1,t_2)=(p(t_1)-p(t_2))/(x_1-x_2)$.
\end{theorem}

\begin{proof}
Expand the alternant in \eqref{eq:O-alt-imrn} along the first row. The cofactor obtained by deleting the $\nu$-th column is
\begin{align*}
 (-1)^{1+\nu}\left(\prod_{\substack{1\le i\le r\\i\ne\nu}}p(t_i)\right)
 V_x(t_1,\ldots,\widehat{t_\nu},\ldots,t_r).
\end{align*}
Dividing by $V_x$ gives \eqref{eq:O-bary-imrn}.

For \eqref{eq:O-Toeplitz-imrn}, put $m=r-1$ and
$w_\nu=\prod_{i\ne\nu}(x_\nu-x_i)^{-1}$.
The barycentric formula for divided differences \eqref{eq:dd-bary-imrn} gives
\begin{align*}
 (x^{r-2+j-i}p)[x_1,\ldots,x_r]
 =\sum_{\nu=1}^{r}x_\nu^{m-i}(w_\nu p(t_\nu))x_\nu^{j-1}.
\end{align*}
Thus the matrix in \eqref{eq:O-Toeplitz-imrn} factors as $ADB$, where
\begin{align*}
 A_{i,\nu}=x_\nu^{m-i},\qquad
 D=\operatorname{diag}(w_1p(t_1),\ldots,w_rp(t_r)),\qquad
 B_{\nu,j}=x_\nu^{j-1}.
\end{align*}
Cauchy--Binet is indexed by the complements of one element. The term indexed by the complement of $\nu$ is
\begin{align*}
 \frac{\prod_{\substack{1\le i\le r\\ i\ne\nu}}p(t_i)}{\prod_{\substack{1\le i\le r\\ i\ne\nu}}(x_i-x_\nu)}.
\end{align*}
Its sum is \eqref{eq:O-bary-imrn}. This proves \eqref{eq:O-Toeplitz-imrn}. Since every entry of the determinant lies in $A_{H,r}$, the operator is regular. Symmetry follows because both the alternant and $V_x$ change by the same sign under a permutation of the variables.

Finally, insert the barycentric identity \eqref{eq:O-bary-imrn} for the two operators of rank \(r-1\) into the right hand side of \eqref{eq:O-rec-imrn}. The terms omitting $1$ or $2$ give the corresponding terms of \eqref{eq:O-bary-imrn}. For an omitted index $\nu\ge3$, the coefficient simplifies as
\begin{align*}
 \frac1{x_1-x_2}\left(\frac1{x_2-x_\nu}-\frac1{x_1-x_\nu}\right)
 =\frac1{(x_1-x_\nu)(x_2-x_\nu)}.
\end{align*}
This gives the remaining barycentric terms.
\end{proof}

\subsection{Explicit formula and reduction modulo \texorpdfstring{$p$}{p}}

This subsection works with the basis functions. The
divided-difference kernel
\begin{equation}\label{eq:K-p-def}
 K_p(u,v)=\frac{p(u)-p(v)}{x(u)-x(v)}
\end{equation}
is used throughout. We first put the explicit formula
\eqref{eq:factorial-dd-imrn} into the determinant and record the resulting
finite expression. For each chosen index $\nu$, we compare the higher tensor with products of values of $K_p$ modulo $p(t_\nu)$. All these products have the same image in the quotient by $p(t_1),\ldots,p(t_r)$.

For $p=\sum_{m=0}^{n}\alpha_mT_m$ and $s\ge0$, use the coefficients
$C_s(k,m)$ of \eqref{eq:Cpath-imrn} and the functions $B_{\ell,m}$ of
\eqref{eq:B-ad-imrn}. The following formula computes the entries of
\eqref{eq:O-Toeplitz-imrn} without additional notation.

\begin{proposition}[Finite combinatorial formula]
\label{thm:O-combinatorial-revised}
For $r\ge2$ and $s\ge0$,
\begin{equation*}
 \begin{split}
 (x^sp)[x_1,\ldots,x_r]
 ={}&\sum_{m=0}^{n}\alpha_m\sum_{k\ge0}C_s(k,m)\\
 &\quad\cdot\sum_{\substack{m_0,\ldots,m_{k-1}\ge0\\
                  m_0+\cdots+m_{k-1}=r-1}}
 \prod_{\ell=0}^{k-1}
 B_{\ell,m_\ell}
 (x_{s_\ell+1},\ldots,x_{s_\ell+m_\ell+1}),
 \end{split}
\end{equation*}
where $s_\ell=m_0+\cdots+m_{\ell-1}$.
Substitution in \eqref{eq:O-Toeplitz-imrn} gives a finite formula for
$\OO_p^{[r]}$. All denominators are products of the units
$1+c_\ell x_i$.
\end{proposition}

\begin{proof}
By Corollary~\ref{cor:iterated-x-imrn},
$x^sp=\sum_m\alpha_m\sum_k C_s(k,m)T_k$.
Apply the $r$-point divided difference and then
\eqref{eq:factorial-dd-imrn}. For $k=0$, the innermost sum is empty
when $r\ge2$, in agreement with the vanishing divided difference of
$T_0=1$. Every nonzero step in \eqref{eq:Cpath-imrn} has length at most
$N$, since a longer step contains a factor $c_{jN}=0$.
Thus $C_s(k,m)=0$ for $k>m+sN$, so the sums are finite.
Formula~\eqref{eq:B-ad-imrn} gives the assertion about denominators.
\end{proof}

For $r=3$, the correction term is a second divided difference. The difference
between the product $K_p(t_1,t_3)K_p(t_2,t_3)$ and the higher Weil
operator is
\begin{align*}
 K_p(t_1,t_3)K_p(t_2,t_3)-\OO_p^{[3]}
 =
 p(t_3)\,p[x_1,x_2,x_3].
\end{align*}

The next proposition gives a congruence modulo $p(t_\nu)$ for every chosen index $\nu$. In the full quotient $(A_H/(p))^{\otimes_H r}$, every such product of values of $K_p$ has the same image as $\OO_p^{[r]}$. This is the comparison used in the residue construction of Section~\ref{sec:pairing}.

\begin{proposition}[Congruence at a chosen index]
\label{prop:anchored-congruence-imrn}
Let $r\ge2$. For every \(\nu\in\{1,\ldots,r\}\), there exists
\(R_{p,\nu}^{[r]}\in A_{H,r}\) with
\begin{equation}\label{eq:anchored-exact-imrn}
 \OO_p^{[r]}
 =\prod_{\substack{1\le i\le r\\ i\ne\nu}}K_p(t_i,t_\nu)-p(t_\nu)R_{p,\nu}^{[r]}.
\end{equation}
In particular, one has the congruence
\begin{equation*}
 \OO_p^{[r]}\equiv\prod_{\substack{1\le i\le r\\ i\ne\nu}}K_p(t_i,t_\nu)\pmod{p(t_\nu)A_{H,r}}.
\end{equation*}
\end{proposition}
\begin{proof}
If $p=0$, both the tensor and the displayed product vanish, and we take $R_{p,\nu}^{[r]}=0$. Assume $p\ne0$. It is enough to take $\nu=r$. Put
\begin{align*}
 P=\prod_{i=1}^{r-1}K_p(t_i,t_r),
 \qquad
 V'=\prod_{1\le i<j\le r-1}(x_j-x_i),
 \qquad
 A=\Alt(1,p,xp,\ldots,x^{r-2}p).
\end{align*}
Let
$A_0=(-1)^{r-1}V'\prod_{i=1}^{r-1}p(t_i)$.
Both
\begin{align*}
 \frac{\prod_{i<r}(p(t_i)-p(t_r))-\prod_{i<r}p(t_i)}{p(t_r)}
 \quad\text{and}\quad
 \frac{A-A_0}{p(t_r)}
\end{align*}
belong to $A_{H,r}$. Define
\begin{align*}
 M=(-1)^{r-1}V'
 \frac{\prod_{i<r}(p(t_i)-p(t_r))-\prod_{i<r}p(t_i)}{p(t_r)}
 -\frac{A-A_0}{p(t_r)}.
\end{align*}
A direct multiplication gives
\begin{equation}\label{eq:VPM-imrn}
 V_x(P-\OO_p^{[r]})=p(t_r)M.
\end{equation}
Fix $i<j$ and pass to $A_{H,r}/(x_j-x_i)$. This quotient is a domain. There $p(t_j)=p(t_i)$ and the alternant has two equal columns. Hence both sides that define $V_xP-A$ vanish, and \eqref{eq:VPM-imrn} gives $p(t_r)M=0$. The image of $p(t_r)$ is nonzero, so $M=0$ in the quotient. Thus every prime factor $x_j-x_i$ of $V_x$ divides $M$. Since $A_{H,r}$ is a UFD, $V_x\mid M$. Write $M=V_xR_{p,r}^{[r]}$ in \eqref{eq:VPM-imrn}. This proves \eqref{eq:anchored-exact-imrn} for $\nu=r$. Symmetry gives every other choice of index $\nu$.
\end{proof}

\section{Residue duality, motives and the Weil pairing}\label{sec:motive-Weil}

Let $X/\Fq$ be a smooth projective geometrically integral curve and $\infty$ a closed point. Put
\begin{align*}
 A_X=H^0(X\setminus\{\infty\},\mathcal O_X),\qquad
 \Omega_X=\Omega_{A_X}.
\end{align*}
The ring $A_X$ is Dedekind. We recall residue duality and its Casimir tensor, then describe the torsion and the finite motive of a Drinfeld module over an $A_X$-field, and construct the Weil pairing with values in the differential module. A geometric realization of finite residue Casimir tensors is given in \cite{HuYu2026}.

\subsection{Residue duality and its Casimir tensor}

For a closed point $v$, write $\operatorname{Res}_v=\Tr_{\kappa(v)/\Fq}\circ\operatorname{res}_v$. The global residue theorem gives
\begin{equation*}
 \sum_{v\in X^{(1)}}\operatorname{Res}_v(\eta)=0
 \qquad(\eta\in\Omega_{\Frac(A_X)/\Fq}).
\end{equation*}
For a nonzero proper ideal $\aideal\subset A_X$, put
$B_{\aideal}=A_X/\aideal$ and
$\Omega_{\aideal}=\aideal^{-1}\Omega_X/\Omega_X$.

\begin{definition}[Residue pairing]\label{def:general-residue-pairing}
For \(\eta\in\aideal^{-1}\Omega_X\) and \(b\in A_X\), the residue pairing is
\begin{equation}\label{eq:general-residue-pairing}
 \langle\bar\eta,\bar b\rangle_{\aideal}
 =\operatorname{Res}_{\infty}(b\eta).
\end{equation}
\end{definition}

The following classical duality is used in van der Heiden's determinant construction \cite{vanderHeiden2004}.

\begin{proposition}[Residue duality]\label{thm:general-residue-duality}
The pairing \eqref{eq:general-residue-pairing} is well defined,
\(A_X\)-balanced, and perfect. Equivalently,
\begin{align*}
 \Omega_{\aideal}\xrightarrow{\ \sim\ }
 \Hom_{\Fq}(B_{\aideal},\Fq),
 \qquad
 \bar\eta\longmapsto
 \bigl(\bar b\mapsto\langle\bar\eta,\bar b\rangle_{\aideal}\bigr),
\end{align*}
is an isomorphism of \(B_{\aideal}\)-modules. Moreover,
\begin{equation*}
 \Omega_{\aideal}
 \simeq
 (\aideal^{-1}\Omega_X)\otimes_{A_X}B_{\aideal}.
\end{equation*}
Consequently, \(\Omega_{\aideal}\) is an invertible
\(B_{\aideal}\)-module. Since \(B_{\aideal}\) is Artinian
semilocal, this module is free of rank one, but it has no preferred
generator in general.
\end{proposition}

\begin{proof}
Write $\aideal=\prod_v\mathfrak m_v^{h_v}$.
If $b\in\aideal$ and $\eta\in\aideal^{-1}\Omega_X$, then $b\eta$
is regular on $X\setminus\{\infty\}$. Its residue at $\infty$ is zero
by the global residue theorem. The same argument applies when a form
in $\Omega_X$ is added to $\eta$. This proves well-definedness, and
$A_X$-balancing follows from $(ab)\eta=a(b\eta)$.

At a point $v$ in the support of $\aideal$, choose a local parameter
$z_v$. Since $X$ is smooth and $\kappa(v)/\Fq$ is separable, completion
identifies the two local factors with
\[
 \kappa(v)[[z_v]]/(z_v^{h_v}),\qquad
 z_v^{-h_v}\kappa(v)[[z_v]]\,dz_v/
             \kappa(v)[[z_v]]\,dz_v.
\]
The local pairing is the negative trace to $\Fq$ of the coefficient
of $z_v^{-1}dz_v$ in $b\eta$. For $0\le i,j<h_v$,
\[
 \operatorname{res}_v(z_v^i z_v^{-j-1}\,dz_v)=\delta_{ij}.
\]
The field trace pairing on $\kappa(v)$ is nondegenerate, so the local
pairing is perfect.
The Chinese remainder theorem and
$\operatorname{Res}_\infty=-\sum_{v\ne\infty}\operatorname{Res}_v$
give the asserted global duality.

Finally,
$\aideal(\aideal^{-1}\Omega_X)=\Omega_X$, so tensoring
$\aideal^{-1}\Omega_X$ with $A_X/\aideal$ gives $\Omega_\aideal$.
The former is an invertible $A_X$-module. Its reduction is therefore
invertible over $B_\aideal$, and every invertible module over an
Artinian semilocal ring is free of rank one.
\end{proof}

For a field extension $k/\Fq$, put $B_{\aideal,k}=B_{\aideal}\otimes_{\Fq}k$, and write $\langle\cdot,\cdot\rangle_{\aideal,k}$ for the scalar extension of the residue pairing \eqref{eq:general-residue-pairing}. A generator $\omega$ of the $B_{\aideal,k}$-module $\Omega_{\aideal}\otimes_{\Fq}k$ defines the functional $\lambda_\omega(b)=\langle\omega,b\rangle_{\aideal,k}$. The form $(a,b)\mapsto\lambda_\omega(ab)$ is symmetric and nondegenerate. The following product of Casimir tensors expands to the higher Casimir tensor of Hu--Huang \cite[Theorem~5.9, Definition~5.10 and equation~(122)]{HuHuang2024}.

\begin{definition}[Casimir tensor and higher Casimir tensor]
Let \(e_1,\ldots,e_s\) be a \(k\)-basis of
\(B_{\aideal,k}\), and let \(e_1^\vee,\ldots,e_s^\vee\) be its dual
basis for \((a,b)\mapsto\lambda_\omega(ab)\). The Casimir tensor and the higher Casimir tensor are
\begin{align}
 C_{\aideal,\omega}
 &=\sum_{i=1}^{s}e_i\otimes e_i^\vee
 \in B_{\aideal,k}^{\otimes_k2},
 \notag\\
 W_{\aideal,\omega}^{[r]}
 &=\prod_{j=1}^{r-1}(C_{\aideal,\omega})_{j,r}
 \in B_{\aideal,k}^{\otimes_k r}
 \qquad(r\ge2).
 \label{eq:general-higher-Casimir}
\end{align}
For $r=1$, set $W_{\aideal,\omega}^{[1]}=1\in B_{\aideal,k}$, using the empty product convention.
\end{definition}

\begin{proposition}[Universal properties of the Casimir tensors]
\label{prop:general-Casimir-properties}
The tensors $C_{\aideal,\omega}$ and $W_{\aideal,\omega}^{[r]}$ are independent
of the chosen basis. For every \(b\in B_{\aideal,k}\),
\begin{align}
 (b\otimes1)C_{\aideal,\omega}
 &=(1\otimes b)C_{\aideal,\omega},
 \label{eq:Casimir-balancing-general}\\
 (\lambda_\omega\otimes1)
 \bigl((b\otimes1)C_{\aideal,\omega}\bigr)
 &=b.
 \label{eq:Casimir-reproduce-general}
\end{align}
Also, \(W_{\aideal,\omega}^{[r]}\) is symmetric and is uniquely
characterized by
\begin{equation}\label{eq:higher-Casimir-characterization}
 (\lambda_\omega^{\otimes r})
 \bigl((b_1\otimes\cdots\otimes b_r)
 W_{\aideal,\omega}^{[r]}\bigr)
 =\lambda_\omega(b_1\cdots b_r)
\end{equation}
for all \(b_1,\ldots,b_r\in B_{\aideal,k}\). In particular,
\begin{equation}\label{eq:higher-full-balancing-general}
 (b_i-b_j)W_{\aideal,\omega}^{[r]}=0
\end{equation}
for every \(b\in B_{\aideal,k}\) and all tensor positions \(i,j\).
If \(\omega'=u\omega\), with \(u\in B_{\aideal,k}^{\times}\), then
\begin{equation}\label{eq:Casimir-scaling-general}
 C_{\aideal,\omega'}
 =(1\otimes u^{-1})C_{\aideal,\omega},
 \qquad
 W_{\aideal,\omega'}^{[r]}
 =(1\otimes\cdots\otimes1\otimes u^{-(r-1)})
 W_{\aideal,\omega}^{[r]}.
\end{equation}
\end{proposition}

\begin{proof}
The perfect form identifies $B_{\aideal,k}\otimes_kB_{\aideal,k}$ with $\End_k(B_{\aideal,k})$ by $a\otimes c\mapsto(z\mapsto a\lambda_\omega(cz))$. The Casimir tensor maps to the identity, proving basis independence and \eqref{eq:Casimir-reproduce-general}. Multiplication by $b$ on either side gives \eqref{eq:Casimir-balancing-general}. Transposition gives symmetry.

Expanding the product \eqref{eq:general-higher-Casimir} and applying reproduction successively proves \eqref{eq:higher-Casimir-characterization}. Since $\lambda_\omega^{\otimes r}$ is perfect, this characterization is unique. Its right hand side is symmetric and unchanged when a factor $b$ is moved between two arguments, proving symmetry and \eqref{eq:higher-full-balancing-general}. Finally, $\lambda_{u\omega}(z)=\lambda_\omega(uz)$, so the new dual basis is $(u^{-1}e_i^\vee)$. This proves both scaling formulas.
\end{proof}
By \eqref{eq:higher-full-balancing-general}, the unit in \eqref{eq:Casimir-scaling-general} may be placed in any tensor factor.

\subsection{Torsion and motive quotient}

Let $L$ be an $A_X$-field whose characteristic is coprime to $\aideal$, let $\bar L=L^{\mathrm{sep}}$, and let $\varphi:A_X\to L\{\tau\}$ be a Drinfeld module of rank $r$, where $\tau$ denotes the $q$-Frobenius operator on $L$, so that $\tau c=c^q\tau$ for $c\in L$. Write $\varphi[\aideal]$ for the $\aideal$-torsion module, consisting of the common zeros of $\varphi_b$ for all $b\in\aideal$. Put $M_\varphi=L\{\tau\}$, with right action $m\cdot a=m\circ\varphi_a$, and $M_{\varphi,\aideal}=M_\varphi/\aideal M_\varphi$. We use the scalar extension
\begin{align*}
 \Omega_{\aideal,\bar L}&=\Omega_{\aideal}\otimes_{\Fq}\bar L.
\end{align*}
The duality below retains $\Omega_{\aideal}$ without choosing a generator.

\begin{lemma}[Motive and torsion modules]\label{lem:finite-motive-structure}
The $A_X\otimes_{\Fq}L$-module $M_\varphi$ is locally free of rank $r$. For $s=\dim_{\Fq}B_{\aideal}$, its quotient $M_{\varphi,\aideal}$ is free of rank $r$ over $B_{\aideal}\otimes_{\Fq}L$ and has $L$-dimension $rs$. If $\aideal$ is away from the characteristic, then $\varphi[\aideal]\simeq B_{\aideal}^r$, and the linearization of $\tau$ on the finite motive is invertible.
\end{lemma}

\begin{proof}
Choose nonconstant $a\in A_X$ and put $m=\dim_{\Fq}A_X/(a)$. The finite map defined by $a$ makes $A_X$ free of rank $m$ over $\Fq[a]$, and $\deg_\tau\varphi_a=rm$. Reduction by its leading term gives the $L[a]$-basis $1,\tau,\ldots,\tau^{rm-1}$ of $M_\varphi$, and independence follows from the distinct $\tau$-degrees modulo $rm$.

The ring $C=A_X\otimes_{\Fq}L$ is Dedekind, since $X$ is smooth and geometrically integral. Every nonzero $c\in C$ divides a nonzero element of $L[a]$. Apply the characteristic polynomial of multiplication by $c$ in the finite free domain extension $C/L[a]$, whose determinant is nonzero. Thus $M_\varphi$ is torsion free and finite over $C$, hence projective. Comparing ranks over $L[a]$ gives rank $r$. Its reduction is projective of constant rank $r$ over the Artinian semilocal ring $B_{\aideal}\otimes_{\Fq}L$, hence free.

The left ideal $\aideal M_\varphi$ has a monic generator $P_{\aideal}$, obtained by choosing least $\tau$-degree and using left division. The quotient dimension gives $\deg_\tau P_{\aideal}=rs$. Choose $a\in\aideal$ with nonzero characteristic image. Since $\varphi_a\in\aideal M_\varphi$, write $\varphi_a=Q P_{\aideal}$ with $Q\in M_\varphi$. The nonzero coefficient of $\tau^0$ in $\varphi_a$ implies the same for $P_{\aideal}$. Hence $P_{\aideal}$ is separable, with roots exactly the common zeros of all $\varphi_b$, $b\in\aideal$. Thus $\#\varphi[\aideal]=q^{rs}$ and these points lie in $L^{\rm sep}$.

For $\mathfrak m^h$ away from the characteristic, this count gives length $rh$ over $(A_X)_{\mathfrak m}$. The submodule killed by $\mathfrak m$ has dimension $r$ over $A_X/\mathfrak m$, so there are $r$ cyclic summands. Each has length at most $h$, hence exactly $h$. The Chinese remainder theorem yields $\varphi[\aideal]\simeq B_{\aideal}^r$. Finally, in the basis $1,\tau,\ldots,\tau^{rs-1}$ modulo $P_{\aideal}$, the linearized $\tau$-matrix is a companion matrix with nonzero determinant because the coefficient of $\tau^0$ in $P_{\aideal}$ is nonzero.
\end{proof}

\begin{proposition}[Torsion and motive duality]
\label{thm:torsion-motive-duality-general}
For \(c\in\bar L\), \(\mu\in\varphi[\aideal]\), and
\(m\in M_\varphi\), there is a unique differential class
\(h_{\varphi,\aideal}(c\otimes\mu)(\bar m)\) satisfying
\begin{equation}\label{eq:h-mu-defining-property}
 \left\langle
 h_{\varphi,\aideal}(c\otimes\mu)(\bar m),\bar b
 \right\rangle_{\aideal,\bar L}
 =c\,m\bigl(\varphi_b(\mu)\bigr)
 \qquad(\bar b\in B_{\aideal}).
\end{equation}
These values form a canonical, Galois-equivariant,
\(B_{\aideal}\otimes_{\Fq}\bar L\)-linear isomorphism
\begin{align*}
 h_{\varphi,\aideal}:
 \bar L\otimes_{\Fq}\varphi[\aideal]
 \xrightarrow{\ \sim\ }
 \Hom_{B_{\aideal}\otimes_{\Fq}\bar L}
 \left(M_{\varphi,\aideal}\otimes_L\bar L,
 \Omega_{\aideal,\bar L}\right).
\end{align*}
\end{proposition}

\begin{proof}
The right hand side of \eqref{eq:h-mu-defining-property} depends only on $\bar b$ and $\bar m$, since $\aideal$ annihilates $\mu$, and for $a\in\aideal$ one has $(m'\circ\varphi_a)(\varphi_b(\mu))=m'(\varphi_{ab}(\mu))=0$. Residue duality gives the unique differential class. Balancing proves linearity in the motive and torsion variables, and the defining formula proves Galois equivariance.

Choose an $\Fq$-basis $\mu_1,\ldots,\mu_{rs}$ of $\varphi[\aideal]$. If $\sum_jc_j\otimes\mu_j$ is in the kernel, taking $\bar b=1$ and $m=1,\tau,\ldots,\tau^{rs-1}$ gives $\sum_jc_j\mu_j^{q^k}=0$ for $0\le k<rs$. The Moore determinant is nonzero, so every $c_j$ vanishes. Both sides have dimension $rs$ over $\bar L$ by Lemma~\ref{lem:finite-motive-structure} and the rank-one property of $\Omega_{\aideal,\bar L}$. Hence the map is an isomorphism.
\end{proof}

\begin{remark}[Frobenius compatibility]
Write $F_V(c\otimes\mu)=c^q\otimes\mu$ and $F_\Omega(\eta\otimes c)=\eta\otimes c^q$ for the coefficient Frobenius maps, and let $\tau_M$ be left multiplication by $\tau$ on the finite motive. The compatibility identity is
\begin{equation}\label{eq:h-Frobenius-compatibility}
 h_{\varphi,\aideal}(F_Vv)(\tau_Mm)
 =F_\Omega\bigl(h_{\varphi,\aideal}(v)(m)\bigr).
\end{equation}
For $v=c\otimes\mu$ and $m\in M_\varphi$, pairing with $b\in B_\aideal$ reduces this to $c^q m(\varphi_b(\mu))^q=(c\,m(\varphi_b(\mu)))^q$. Residue duality and semilinearity give the general identity.
\end{remark}

\subsection{Van der Heiden's determinant pairing with values in \texorpdfstring{$\Omega_{\aideal}$}{the differential module}}

We use the determinant construction of \cite{vanderHeiden2004}.
Fix a rank-one Drinfeld module $\psi$ and a specified $\tau$-compatible
isomorphism
\[
 \bigwedge^r_{A_X\otimes_{\Fq}L}M_\varphi\simeq M_\psi.
\]
All pairing statements below are relative to these data.
Reduction modulo $\aideal$ gives
\begin{equation*}
 \delta_{\varphi,\aideal}:
 \bigwedge_{B_{\aideal}\otimes_{\Fq}L}^{r}M_{\varphi,\aideal}
 \xrightarrow{\ \sim\ }M_{\psi,\aideal}.
\end{equation*}
For a rank-$r$ projective module $M$ and an invertible module $\Omega$ over a commutative ring $B$, the exterior Hom isomorphism
\begin{equation}\label{eq:exterior-Hom-isomorphism}
 \bigwedge^{r}\Hom_B(M,\Omega)
 \xrightarrow{\ \sim\ }
 \Hom_B(\det M,\Omega^{\otimes r})
\end{equation}
sends $f_1\wedge\cdots\wedge f_r$ to the map taking $m_1\wedge\cdots\wedge m_r$ to $\det(f_i(m_j))$. This gives van der Heiden's pairing in the residue realization.

\begin{proposition}
\label{thm:Weil-isomorphism}
For $\aideal$ away from the characteristic, the determinant Weil pairing is a Galois-equivariant isomorphism
\begin{equation}\label{eq:Omega-Weil-general}
 \operatorname{Weil}^{\Omega}_{\aideal}:
 \bigwedge_{B_{\aideal}}^r\varphi[\aideal]
 \xrightarrow{\ \sim\ }
 \psi[\aideal]\otimes_{B_{\aideal}}\Omega_{\aideal}^{\otimes(r-1)}.
\end{equation}
It is relative to the chosen $\tau$-compatible determinant realization and natural under extension of the definition field.
\end{proposition}

\begin{proof}
Extend scalars to $C=B_{\aideal}\otimes_{\Fq}\bar L$, and compose $\bigwedge^r h_{\varphi,\aideal}$ with the exterior Hom isomorphism \eqref{eq:exterior-Hom-isomorphism} and the determinant realization $\delta_{\varphi,\aideal}$. Since $M_\psi$ has rank one, the target is
\begin{align*}
 \Hom_C(M_{\psi,\aideal,\bar L},\Omega_{\aideal,\bar L})
 \otimes_C\Omega_{\aideal,\bar L}^{\otimes(r-1)}.
\end{align*}
Applying $h_{\psi,\aideal}^{-1}$ to the first factor identifies the scalar extensions in \eqref{eq:Omega-Weil-general}. Write $F_C=\operatorname{id}\otimes(c\mapsto c^q)$ and $F_C^*P=C\otimes_{F_C,C}P$. Lemma~\ref{lem:finite-motive-structure} gives the invertible linearization $\tau_M^{\rm lin}:F_C^*(M_{\varphi,\aideal}\otimes_L\bar L)\to M_{\varphi,\aideal}\otimes_L\bar L$, and the analogous $F_\Omega^{\rm lin}$ is invertible because $\Omega_{\aideal,\bar L}$ is extended from $\Fq$. By \eqref{eq:h-Frobenius-compatibility}, Frobenius on the Hom module is $F_{\rm Hom}(f)=F_\Omega^{\rm lin}\circ F_C^*f\circ(\tau_M^{\rm lin})^{-1}$, which is defined even though $\bar L$ need not be perfect, and $h_{\varphi,\aideal}\circ F_V=F_{\rm Hom}\circ h_{\varphi,\aideal}$. Together with the exterior determinant formula and the $\tau$-compatibility of $\delta_{\varphi,\aideal}$, this shows that the isomorphism commutes with coefficient Frobenius. Both sides are scalar extensions of finite-dimensional $\Fq$-spaces. In an $\Fq$-basis, fixed points satisfy $c^q=c$ on each coefficient. Taking fixed points gives precisely \eqref{eq:Omega-Weil-general}. All maps are Galois equivariant and compatible with extension of $L$.
\end{proof}

For scalar coordinates, choose a finite separable extension $k/L$ in $\bar L$ containing the $\aideal$-torsion of $\varphi$ and $\psi$ and defining the motive bases below. For a generator $\omega$ of the $B_{\aideal,k}$-module $\Omega_{\aideal}\otimes_{\Fq}k$, writing $\nu\otimes b\omega^{\otimes(r-1)}\mapsto b\cdot\nu$ for the corresponding trivialization, define
\begin{align*}
 \operatorname{Weil}_{\aideal,\omega}
 =(\nu\otimes b\omega^{\otimes(r-1)}\mapsto b\cdot\nu)\circ(\operatorname{Weil}^{\Omega}_{\aideal}\otimes k).
\end{align*}
Let $\mathbf m=(m_1,\ldots,m_r)$ be a $B_{\aideal,k}$-basis of the finite motive, and put $m_\psi=\delta_{\varphi,\aideal}(m_1\wedge\cdots\wedge m_r)$. For $c=\sum_\nu b_\nu\otimes\gamma_\nu\in B_{\aideal,k}$, its torsion action is
$c\diamond_{\varphi,k}(1\otimes\mu)=\sum_\nu\gamma_\nu\otimes\varphi_{b_\nu}(\mu)$, extended $k$-linearly. This is independent of the presentation of $c$. For $m\in M_\varphi$ and $c\in k$, evaluation on the formal
scalar extension means
\[
 m(c\otimes\mu)=c\,m(\mu).
\]
Extend this formula $k$-linearly in $m$ as well. In particular, the
scalar in the first tensor factor is kept outside the additive
polynomial. This convention is used in the determinant below. For $\Xi=\sum_\kappa c_{\kappa,1}\otimes\cdots\otimes c_{\kappa,r}\in B_{\aideal,k}^{\otimes_kr}$, define
\begin{equation*}
 \operatorname{DEv}_{\varphi,\mathbf m}
 (\Xi;\mu_1,\ldots,\mu_r)
 =\sum_\kappa\det\!\left(
 m_i\!\left(c_{\kappa,j}\diamond_{\varphi,k}(1\otimes\mu_j)\right)
 \right)_{1\le i,j\le r}.
\end{equation*}
Multilinearity makes this independent of the tensor presentation. We also write $m_\psi$ for its induced $k$-linear evaluation on determinant torsion. The following is the dual basis formula of \cite[Theorem~5.9 and equation~(122)]{HuHuang2024}, expressed through the residue Casimir tensor.

\begin{proposition}[General dual basis formula]
\label{thm:general-dual-basis-formula}
For \(\mu_1,\ldots,\mu_r\in\varphi[\aideal]\), put
\(\boldsymbol\mu=(\mu_1,\ldots,\mu_r)\). Then
\begin{align*}
 m_\psi\!
 \left(\operatorname{Weil}_{\aideal,\omega}
 (\mu_1,\ldots,\mu_r)\right)
 =
 \operatorname{DEv}_{\varphi,\mathbf m}
 \left(W_{\aideal,\omega}^{[r]};
 \mu_1,\ldots,\mu_r\right).
\end{align*}
The formula is independent of every basis used to write the Casimir tensor.
If \(\omega'=u\omega\), with
\(u\in B_{\aideal,k}^{\times}\), then the pairing with values in $\Omega_{\aideal}$ is unchanged and
\begin{equation}\label{eq:scalar-Weil-scaling-correct}
 \operatorname{Weil}_{\aideal,\omega'}
 =u^{-(r-1)}\cdot
 \operatorname{Weil}_{\aideal,\omega}
 \quad\text{in }k\otimes_{\Fq}\psi[\aideal].
\end{equation}
If $u^{-(r-1)}\in k^\times$, in particular if $u\in k^\times$, this factor may be pulled through the $k$-linear evaluation map $m_\psi$. For a general unit $u$, the factor acts on the module inside the evaluation map as in \eqref{eq:scalar-Weil-scaling-correct}.
\end{proposition}

\begin{proof}
Let $(b_\ell)$ and $(b_\ell^\vee)$ be dual $k$-bases of $B_{\aideal,k}$ for $\lambda_\omega$, and write $h_{\varphi,\aideal}(1\otimes\mu_j)(m_i)=u_{ij}\omega$. By \eqref{eq:h-mu-defining-property}, $\lambda_\omega(u_{ij}c)=m_i(c\cdot\mu_j)$, so
\begin{equation}\label{eq:uij-dual-expansion}
 u_{ij}=\sum_{\ell=1}^{s}m_i(b_\ell\cdot\mu_j)b_\ell^\vee.
\end{equation}
Put $Q_{\boldsymbol\mu}=\det(u_{ij})$. Exterior Hom gives
$h_{\psi,\aideal}(\operatorname{Weil}_{\aideal,\omega}(\boldsymbol\mu))(m_\psi)=Q_{\boldsymbol\mu}\omega$. Pairing with $1$ yields
\begin{equation*}
 m_\psi\!\left(\operatorname{Weil}_{\aideal,\omega}
 (\mu_1,\ldots,\mu_r)\right)=\lambda_\omega(Q_{\boldsymbol\mu}).
\end{equation*}
For a permutation $\sigma$, set $\ell_j(c)=m_{\sigma(j)}(c\cdot\mu_j)$ and $v_j=\sum_\ell\ell_j(b_\ell)b_\ell^\vee$. Then $\ell_j(c)=\lambda_\omega(v_jc)$, and \eqref{eq:higher-Casimir-characterization} gives
\begin{align*}
 \lambda_\omega(v_1\cdots v_r)
 =\sum_\kappa\prod_j\ell_j(c_{\kappa,j}),
 \qquad W_{\aideal,\omega}^{[r]}
 =\sum_\kappa c_{\kappa,1}\otimes\cdots\otimes c_{\kappa,r}.
\end{align*}
Sum over $\sigma$ with determinant signs and use \eqref{eq:uij-dual-expansion}. The left side becomes $\lambda_\omega(Q_{\boldsymbol\mu})$, and the right side becomes $\operatorname{DEv}_{\varphi,\mathbf m}(W_{\aideal,\omega}^{[r]};\boldsymbol\mu)$, proving the formula.

Basis independence follows from Proposition~\ref{prop:general-Casimir-properties}. If $\omega'=u\omega$, the two trivializations differ by the unit $u^{-(r-1)}$, giving \eqref{eq:scalar-Weil-scaling-correct}. The Casimir tensor scaling law gives the corresponding evaluation identity.
\end{proof}

\section{The residue pairing and the Weil pairing on the punctured line}\label{sec:pairing}

This section specializes the general construction to $\A$. We compute the residue dual basis and expand the kernel $K_p$. We then identify the higher Casimir tensor with the higher Weil operator and with the determinant Weil pairing.

\subsection{The quotient, the dual basis, and the Casimir tensor}

Let \(p\in\A\setminus\Fq\) have exact pole order \(d\ge1\) at
\(P_\rho\), and put \(n=Nd=\dim_{\Fq}(\A/p\A)\). Its expansion in $(T_k)$ is
\begin{align*}
 p=\sum_{k=0}^{n}\alpha_kT_k,
 \qquad
 \alpha_n\ne0.
\end{align*}
The quotient is \(B_p=A_H/(p)\), and the representative functions are
\begin{align*}
 e_i=T_i^{(-1)}
 \qquad(0\le i<n).
\end{align*}

The divisor $D_p=d(P_0+\cdots+P_{N-1})-P_0$ equals
$D_{n-1}^{(-1)}$. The periodic list $\zeta^{(-1)},\ldots,\zeta^{(-(n-1))}$
contains $P_0$ exactly $d-1$ times and every other split point exactly
$d$ times. By Proposition~\ref{prop:factorial-basis-imrn} with
$s=-1$, the functions $e_0,\ldots,e_{n-1}$ form an $H$-basis of
$L(D_p)$. The following proposition carries this basis to the quotient
$B_p$.

\begin{proposition}\label{prop:quotient-basis-imrn}
The classes of \(e_0,\ldots,e_{n-1}\) form an \(H\)-basis of \(B_p\).
\end{proposition}

\begin{proof}
The quotient map \(L(D_p)\to B_p\) is injective. If \(F\) lies in its
kernel, then \(F=ph\) for some \(h\in A_H\). At \(P_0\),
\begin{align*}
 \operatorname{ord}_{P_0}(F)\ge -(d-1),
 \qquad
 \operatorname{ord}_{P_0}(p)=-d,
\end{align*}
so \(\operatorname{ord}_{P_0}(h)\ge1\), and \(h\) vanishes at \(P_0\).
For \(s\ne0\), one has \(\operatorname{ord}_{P_s}(F)\ge-d\) and
\(\operatorname{ord}_{P_s}(p)=-d\), so \(\operatorname{ord}_{P_s}(h)\ge0\).
Since \(h\in A_H\) is regular away from the \(P_s\), it is a global regular
function on \(\mathbf P^1_H\) vanishing at \(P_0\), hence \(h=0\). Finally,
the zero divisor of \(p\) has degree \(n\), so \(\dim_H B_p=n\). Therefore
the injective map is an isomorphism.
\end{proof}

Since $A_H$ is a localization of $H[x]$, one has
$\Omega_{A_H}=A_H\,dx$. Put
$\Omega_p=p^{-1}\Omega_{A_H}/\Omega_{A_H}$ and choose its
generator $\omega_p=-dx/p$.
Multiplication by \(\omega_p\) identifies \(B_p\) with \(\Omega_p\).
For \(f,g\in A_H\), the residue pairing is
\begin{equation}\label{eq:res-pair-imrn}
 \langle\bar f\omega_p,\bar g\rangle_H
 =
 \sum_{s=0}^{N-1}\operatorname{Res}_{P_s}(fg\omega_p).
\end{equation}
The associated residue functional is
\begin{align*}
 \ell_p:B_p\longrightarrow H,
 \qquad
 \ell_p(\bar F)=\sum_{s=0}^{N-1}\operatorname{Res}_{P_s}(F\omega_p).
\end{align*}
It satisfies
\(\langle\bar f\omega_p,\bar g\rangle_H=\ell_p(\bar f\bar g)\).

Before computing the dual basis, we verify that \eqref{eq:res-pair-imrn}
depends only on the indicated quotient classes.

\begin{lemma}\label{lem:res-pair-well-defined}
The expression \eqref{eq:res-pair-imrn} depends only on the classes of
\(f\omega_p\) in \(\Omega_p\) and of \(g\) in \(B_p\).
\end{lemma}

\begin{proof}
If \(f\) is replaced by \(f+ph_1\), with \(h_1\in A_H\), the change in the
differential is \(-h_1g\,dx\). This differential is regular away from
\(P_0,\ldots,P_{N-1}\). The sum of its residues at these points is zero by
the global residue theorem. The same argument applies when \(g\) is replaced
by \(g+ph_2\), with $h_2\in A_H$, and when a regular differential is added to \(f\omega_p\).
\end{proof}

To obtain the dual basis, we first isolate the coefficient of \(e_i(v)\) in the expansion of $K_p$. The following functions collect these coefficients.
For \(0\le i<n\), define
\begin{equation}\label{eq:D-i-imrn}
 \DD_i(t)
 =
 \sum_{k=i+1}^{n}
 \alpha_k\frac{T_{k-1}^{(-1)}(t)}{T_i(t)}.
\end{equation}
Formula \eqref{eq:shift-identity-imrn} gives
\begin{equation}\label{eq:D-i-closed-imrn}
 \DD_i(t)
 =
 \frac{1}{1+c_ix(t)}
 \sum_{m=0}^{n-i-1}
 \alpha_{i+m+1}T_m^{(-i-1)}(t).
\end{equation}
This closed form shows that \(\DD_i\in A_H\).

For the kernel $K_p$ of \eqref{eq:K-p-def}, the functions $\DD_i$
collect the coefficients of the basis functions. The next proposition
gives the finite expansion of $K_p$ and identifies the dual basis.

\begin{proposition}\label{thm:dual-Casimir-imrn}
The function \(K_p\) has the finite expansion
\begin{equation}\label{eq:K-dual-exp-imrn}
 K_p(u,v)
 =
 \sum_{i=0}^{n-1}\DD_i(u)e_i(v).
\end{equation}
The coefficients satisfy the duality relations
\begin{equation}\label{eq:dual-pair-imrn}
 \langle\DD_j\omega_p,e_i\rangle_H=\delta_{ij}.
\end{equation}
Thus $(\DD_j\omega_p)$ is the basis of $\Omega_p$ dual to the basis $(e_i)$ of $B_p$ relative to \eqref{eq:res-pair-imrn}. In particular, \eqref{eq:res-pair-imrn} is perfect, and the class of $K_p$ in $B_p\otimes_H B_p$ is the Casimir tensor of $B_p$.
\end{proposition}

\begin{proof}
For \(k\ge1\), the factors defining $T_k$ satisfy the
summation identity
\begin{equation}\label{eq:T-rank2-telescope-imrn}
 \frac{T_k(u)-T_k(v)}{x(u)-x(v)}
 =
 T_{k-1}^{(-1)}(u)
 \sum_{i=0}^{k-1}
 \frac{T_i^{(-1)}(v)}{T_i(u)}.
\end{equation}
Indeed, successive cancellation in the sum gives the identity.
This identity is also the case \(r=2\) of
Proposition~\ref{prop:factorial-dd-imrn}. Multiplying
\eqref{eq:T-rank2-telescope-imrn} by \(\alpha_k\), summing over \(k\),
and collecting the coefficient of \(e_i(v)\) gives
\eqref{eq:K-dual-exp-imrn}.

The reproducing property of \(K_p\) in the
finite \(H\)-algebra \(B_p\) is
\begin{equation}\label{eq:reproducing-imrn}
 (\ell_p\tensor1)\bigl(F(u)K_p(u,v)\bigr)=F(v)
 \qquad(F\in A_H).
\end{equation}
Modulo \(p(v)\), one has
\begin{align*}
 F(u)K_p(u,v)\omega_p(u)
 =
 -\frac{F(u)\,dx(u)}{x(u)-x(v)}.
\end{align*}
Regard this as a differential on
\(\mathbf P^1_H\times\Spec B_p\), relative to \(\Spec B_p\).
It is regular away from the sections associated with the
punctures \(P_0,\ldots,P_{N-1}\) and from the section defined by
\(x(u)=x(v)\). Put \(a=x(v)\in B_p\). Since every \(1+c_sa\) is a unit,
this section is disjoint from the removed punctures. In the relative
parameter \(y=x(u)-a\), one has \(F(a+y)\in B_p[[y]]\). Thus the differential
is \(-F(a+y)\,dy/y\), whose residue is \(-F(a)=-F(v)\).
The relative residue theorem, applied coefficientwise after a choice
of \(H\)-basis of \(B_p\), gives \eqref{eq:reproducing-imrn}, also when
\(B_p\) is nonreduced, since relative residues commute with finite base change.

Taking \(F=e_j\) in \eqref{eq:reproducing-imrn} and inserting
\eqref{eq:K-dual-exp-imrn}, comparison of coefficients in the basis
\(e_0(v),\ldots,e_{n-1}(v)\) gives \eqref{eq:dual-pair-imrn}. Hence the
pairing is perfect.
\end{proof}

\begin{remark}[Constant extension and scalar coordinates]\label{rem:trace-descent-imrn}
The canonical identifications
\[
 \Omega_{(p)}\otimes_{\Fq}H\simeq\Omega_p,
 \qquad (\A/(p))\otimes_{\Fq}H\simeq B_p
\]
identify \eqref{eq:res-pair-imrn} with the scalar extension of
\eqref{eq:general-residue-pairing}. Indeed, the residue trace at
$P_\rho$ becomes the sum of the residues at $P_0,\ldots,P_{N-1}$.
This explains descent of the differential-valued pairing.

The chosen scalar coordinate depends on $\zeta$. More precisely,
\[
 \omega_{p,\zeta^q}
 =\left(\frac{t-\zeta}{t-\zeta^q}\right)^2\omega_{p,\zeta}.
\]
The factor is a unit of $B_p$. Thus
\eqref{eq:Casimir-scaling-general} and
\eqref{eq:scalar-Weil-scaling-correct} give the changes in the tensor
and in the scalar coordinate.
\end{remark}

\subsection{The explicit pairing formula}

This subsection reduces the higher Casimir tensor to the
higher Weil operator. The representative space is
\begin{align*}
 V_p
 =
 L\!\left(d\sum_{s=0}^{N-1}P_s-P_0\right)
 =
 \Span_H\{e_0,\ldots,e_{n-1}\},
\end{align*}
the quotient map \(V_p\to B_p\) is an isomorphism, and its
inverse is denoted \(\operatorname{rem}_{p}\).

\begin{definition}\label{def:pairing-operator-imrn}
Write $\operatorname{rem}_{p,t_r}$ for the remainder map
$\operatorname{rem}_p$ applied in the last variable $t_r$.
For \(r\ge2\), the higher tensor is
\begin{equation}\label{eq:W-pairing-imrn}
 \WW_{p,\zeta}^{[r]}
 =
 \operatorname{rem}_{p,t_r}\!\left(
 \prod_{j=1}^{r-1}K_p(t_j,t_r)
 \right)
 \in
 A_H^{\tensor_H(r-1)}\tensor_H V_p.
\end{equation}
It is the unique element of
$A_H^{\tensor_H(r-1)}\tensor_H V_p$ congruent to the product modulo $p(t_r)$.
\end{definition}

\begin{proposition}\label{thm:pairing-equals-O-imrn}
Let $\OO_{p,\zeta}^{[r]}$ be the higher Weil operator of
\eqref{eq:intro-O-revised}.
For every \(r\ge2\),
\begin{equation}\label{eq:W-equals-O-imrn}
 \WW_{p,\zeta}^{[r]}=\OO_{p,\zeta}^{[r]}.
\end{equation}
Moreover, $\OO_{p,\zeta}^{[r]}\in V_p^{\otimes_H r}$.
\end{proposition}

\begin{proof}
Proposition~\ref{prop:anchored-congruence-imrn} shows that the product
in \eqref{eq:W-pairing-imrn} and \(\OO_{p,\zeta}^{[r]}\) have the same
class modulo \(p(t_r)\). We show that the latter already belongs to
the selected representative space in the last variable.

At \(P_0\), the entries in the last column of the alternant
\(\Alt(1,p,xp,\ldots,x^{r-2}p)\) have pole orders at most
$0,d,d+1,\ldots,d+r-2$.
The factor \(\prod_{i<r}(x_r-x_i)\) in the Vandermonde has pole order
\(r-1\), so the quotient has pole order at most \(d-1\) at \(P_0\).
At \(P_s\), \(s\ne0\), the coordinate \(x_r\) is regular and the numerator has pole order at most \(d\). After specialization at \(P_s\), each \(x_r-x_i\) is a nonzero scalar multiple of \(1+c_sx_i\). So it is a unit in the coefficient ring \(A_H^{\tensor_H(r-1)}\). The pole order is therefore at most \(d\). There are no other poles by
Theorem~\ref{thm:universal-O-imrn}. These pole bounds define the
$H$-linear subspace $V_p\subset A_H$. Since the operator belongs to
$A_{H,r}$, the bounds imply
\[
 \OO_{p,\zeta}^{[r]}
 \in A_H^{\tensor_H(r-1)}\tensor_H V_p.
\]
Uniqueness proves \eqref{eq:W-equals-O-imrn}.
Symmetry gives the same pole bounds in every variable, hence
$\OO_{p,\zeta}^{[r]}\in V_p^{\otimes_H r}$.
\end{proof}

Let $\pideal=(p)$ be a principal prime away from the
characteristic, let $\varphi$ be a rank $r$ Drinfeld $\A$-module, and
fix the determinant data $(\psi,\delta_{\varphi,\pideal})$ of
Section~\ref{sec:motive-Weil}. After the constant extension to $H$,
Proposition~\ref{thm:dual-Casimir-imrn} identifies, for the generator
$\omega_p$ of $\Omega_p$, the Casimir tensor
$C_{\pideal,\omega_p}=\overline{K_p(u,v)}$. Consequently,
\begin{equation}\label{eq:general-W-is-split-O}
 W_{\pideal,\omega_p}^{[r]}
 =\prod_{j=1}^{r-1}\overline{K_p(t_j,t_r)}
 =\overline{\OO}_{p,\zeta}^{[r]},
\end{equation}
where the last equality is Proposition~\ref{thm:pairing-equals-O-imrn}.
Specializing the general dual basis formula to $\pideal$ gives a
determinant evaluation formula for any motive basis. If $1,\tau,\ldots,\tau^{r-1}$ is a basis of the finite motive,
the corresponding determinant is a Moore determinant. We do not need
this additional basis assumption.

The formula does not require a Frobenius basis. Write $\Cinf$ for the completion of an algebraic closure of $K=\Fq(t)$ at $P_\rho$, let $L\subset\Cinf$ be a definition field for $\varphi$, $\psi$, and the chosen determinant realization, and take a finite separable extension $k/(LH)$ containing the $\pideal$-torsion of both modules. Over $B_{\pideal,k}$, choose a basis $\mathbf m=(m_1,\ldots,m_r)$ of the scalar-extended finite motive and put $m_\psi=\delta_{\varphi,\pideal}(m_1\wedge\cdots\wedge m_r)$. With $m_\psi$ as the $k$-linear evaluation of Section~\ref{sec:motive-Weil}, this is the specialization of Proposition~\ref{thm:general-dual-basis-formula}.

\begin{proposition}[A formula for the Weil pairing]
\label{thm:actual-Weil-formula}
For \(\mu_1,\ldots,\mu_r\in\varphi[\pideal]\), the scalar coordinate of the determinant Weil pairing determined by \(\omega_p\) satisfies
\begin{equation}\label{eq:actual-Weil-formula}
 m_\psi\!\left(
 \operatorname{Weil}_{\pideal,\omega_p}
 (\mu_1,\ldots,\mu_r)
 \right)
 =
 \operatorname{DEv}_{\varphi,\mathbf m}
 \left(
 \overline{\OO}_{p,\zeta}^{[r]};
 \mu_1,\ldots,\mu_r
 \right).
\end{equation}
\end{proposition}

\begin{proof}
Apply Proposition~\ref{thm:general-dual-basis-formula} with
\(\omega=\omega_p\) and the chosen motive basis \(\mathbf m\).
Identity \eqref{eq:general-W-is-split-O} replaces the general higher
Casimir tensor by \(\overline{\OO}_{p,\zeta}^{[r]}\), giving
\eqref{eq:actual-Weil-formula}.
\end{proof}

\section{Analytic remainders and prime torsion}
\label{sec:drinfeld}
This section proves an exact division identity for the generating function of the specified normalized rank-one module $\Phi$ modulo a nonconstant element $p$. Its remainder coefficients are explicit operator expressions in an exponential division value. The last coefficient is a nonzero scalar multiple of that value. At a finite principal prime, after adjoining the residue-field element $\lambda$, the special value $\omega_\Phi(\lambda)$ generates the full prime-torsion extension.

\subsection{The normalized rank-one Drinfeld module}
\label{subsec:normalized-dual}

Let $\theta$ be the generic characteristic image of $t$, and
put $\Theta=(\theta-\zeta)^{-1}$ and
$f_\zeta(t)=(t-\zeta)^{-1}-\Theta$. The superscript $(j)$ applies $q^j$-Frobenius to all
coefficients, including $\theta$, and fixes $t$. Explicitly
$f_\zeta^{(j)}(t)=(t-\zeta^{q^j})^{-1}-\Theta^{q^j}$.

Let $\sigma$ act on $H(\theta)$ by $\sigma(\zeta)=\zeta^q$
and $\sigma(\theta)=\theta$. Put $s_0=1$ and
$s_j(t)=\prod_{k=0}^{j-1}f_\zeta^{(k)}(t)$. For $0\ne a\in\A$, let
$v_{P_\rho}(a)$ be the valuation at $P_\rho$ and put
$n(a)=-Nv_{P_\rho}(a)$. Each $a\in\A$ has a finite expansion
\begin{align*}
 a(t)=\sum_{j=0}^{n(a)}b_j(a)s_j(t),\qquad b_j(a)\in H(\theta),
\end{align*}
and its coefficients define the additive polynomial
$\Psi_a=\sum_{j=0}^{n(a)}b_j(a)\tau^j$ of
\cite[Definition~3.7]{HuHuangYau2026}, with $\Psi_0=0$.

Write $\Psi_a^{\sigma^2}=\sum_j\sigma^2(b_j(a))\tau^j$,
where $\sigma^2$ raises the elements of $H$ to their $q^2$-th powers
but fixes $\theta$, and choose $c\in\Cinf^\times$. The
constant $C$ and the normalized module $\Phi$ are then given by
\begin{align}\label{eq:Phi-explicit-normalization}
 c^{q-1}=\frac{\sigma^2(\Theta)}{\Theta},\qquad
 C=c\frac{\sigma(\Theta)}{\Theta},\qquad
 \Phi_a(z)=C^{-1}\Psi_a^{\sigma^2}(Cz).
\end{align}
Thus $C$ is the constant $C_{0,2}$ of
\cite[Definition~6.6]{HuHuangYau2026}, and changing $c$ by an element
of $\Fq^\times$ leaves $\Phi$ unchanged.

In power series, the formal logarithm and its inverse are
\begin{align*}
 \log_\Phi(z)
 &=z+\sum_{m\ge1}
 \frac{z^{q^m}}{\prod_{j=1}^m f_\zeta^{(j)}(\theta)},&
 \exp_\Phi(z)&=\sum_{m\ge0}\frac{z^{q^m}}{D_m^\Phi},
 \quad D_0^\Phi=1,
\end{align*}
where $D_m^\Phi\in\Cinf^\times$ are the denominators of the
exponential coefficients. The logarithm is used here as a formal power
series; the exponential converges on $\Cinf$.
$\Phi_a\exp_\Phi(z)=\exp_\Phi(a(\theta)z)$
\cite[Equations~(134)--(138)]{HuHuangYau2026}.

\begin{proposition}\label{prop:Psi-Phi-comparison}
Choose $\kappa\in\Cinf$ with $\kappa^{q-1}=-\Theta$. The period is
\begin{align}\label{eq:Phi-period-explicit}
 \wt\pi_\Phi
 =\kappa\Theta\prod_{m\ge1}
 \frac{\Theta^{q^m}}{\Theta^{q^m}-\sigma^m(\Theta)}.
\end{align}
Then the exponential lattice is
\begin{equation*}
 \ker\exp_\Phi=\wt\pi_\Phi\iota(\A).
\end{equation*}
\end{proposition}
\begin{proof}
For $N=1$, one has $\A=\Fq[x]$ and $\Phi_x=\Theta+\tau$. The displayed product and lattice are the usual Carlitz period and its kernel \cite{Goss1996}. Assume now that $N\ge2$.
The constant field twist $\Psi^{\sigma^2}$ has exponential
$\exp_{(2)}$ and period
\begin{align*}
 \wt\pi=c\kappa\,\sigma(\Theta)
 \prod_{m\ge1}\frac{\Theta^{q^m}}
 {\Theta^{q^m}-\sigma^m(\Theta)}
\end{align*}
with $\ker\exp_{(2)}=\wt\pi\iota(\A)$
by \cite[Theorem~3.32]{HuHuangYau2026}.
Conjugation gives
$\exp_\Phi(z)=C^{-1}\exp_{(2)}(Cz)$ and
$\wt\pi/C=\wt\pi_\Phi$, proving the assertion.
\end{proof}

\subsection{The generating function}
\label{subsec:generating-function}

Write $|\cdot|_\rho$ for the absolute value on $\Cinf$ that
extends the $P_\rho$-adic valuation, normalized on the branch with
$|\theta-\zeta|_\rho=q^{-1}$. Write $\iota(a)=a(\theta)$ for the
generic characteristic embedding and put
\begin{align*}
 \beta(u)=\frac{u-\zeta^{(1)}}{u-\zeta},\qquad
 \gamma_\zeta=\beta(\theta)^{-1}.
\end{align*}

The analytic arguments below take place on the affinoid domain
\begin{align*}
 \mathbb D_\zeta=\{t\mid |(t-\zeta^{(s)})^{-1}|_\rho\le1\ (0\le s<N)\}.
\end{align*}
It contains the ordinary point $P_\infty$, and
$\mathbb T_\zeta=\mathcal O(\mathbb D_\zeta)$ denotes the algebra
of analytic functions on it. The domain $\mathbb D_\zeta$ is
one-dimensional, and the functions $Y_s=(t-\zeta^{(s)})^{-1}$ generate
$\mathbb T_\zeta$ as an affinoid algebra and satisfy
$Y_s-Y_r=(\zeta^{(s)}-\zeta^{(r)})Y_sY_r$.

For $g\in\A\otimes_{\Fq}H$ and a fixed $\mu\in\Cinf$, the $H$-linear evaluation map is
\begin{equation*}
 g\diamond_{\varphi,H}\mu
 =
 \sum_{\nu=1}^{s}\lambda_\nu\,\varphi_{b_\nu}(\mu),
 \qquad
 g=\sum_{\nu=1}^{s}b_\nu\otimes\lambda_\nu\in\A\otimes_{\Fq}H,
\end{equation*}
extended \(\Fq\)-linearly from \(b\mapsto\varphi_b(\mu)\). It is
\(H\)-linear in \(g\) and satisfies
\((ga)\diamond_{\varphi,H}\mu
=g\diamond_{\varphi,H}\varphi_a(\mu)\) for $a\in\A$.

\begin{definition}[Exponential evaluation map \cite{HuHuangYau2026}]\label{def:exponential-action-revised}
For $z\in\Cinf$, regard $K\otimes_{\Fq}H$ as the copy
$\Fq(u)\otimes_{\Fq}H$ with independent variable $u$. For
$\xi(u)=g(u)/a(u)$, where $0\ne a\in\A$ and $g\in A_H$, set
\begin{align*}
 \EE_{\varphi,z}(\xi(u))=g\diamond_{\varphi,H}
 \exp_\varphi\left(\frac{z}{a(\theta)}\right).
\end{align*}
\end{definition}

Throughout, $z\in\Cinf$ is the input of the exponential,
$u$ is the variable of the rational kernel functions, and $t$ is the
analytic variable. The notation $\EE_{\varphi,z}(\xi(u))$ means that
the evaluation is applied in $u$, with $z$ fixed, and any dependence on $t$
is left unchanged.

\begin{lemma}[The evaluation series]\label{lem:exponential-action-series}
The evaluation map $\EE_{\varphi,z}$ is independent of the
chosen fraction, is $H$-linear in $\xi$, and satisfies
\begin{align}\label{Eq:EEefandds}
 \EE_{\Phi,z}(\xi)=\sum_{m\ge0}\frac{z^{q^m}}{D_m^\Phi}\,
 \xi(\theta^{q^m}),\qquad
 \EE_{\Phi,z}(\xi/a)=\EE_{\Phi,z/a(\theta)}(\xi)
 \quad(0\ne a\in\A).
\end{align}
In $\xi(\theta^{q^m})$, the coefficients in $H$ are fixed, not raised to $q^m$.
\end{lemma}
\begin{proof}
Write the two fractions $\xi=g_1/a_1=g_2/a_2$ of
Definition~\ref{def:exponential-action-revised} with a common
denominator. The functional equation
$\varphi_b\exp_\varphi(V)=\exp_\varphi(b(\theta)V)$ of the
exponential then gives the same value for both and proves additivity
and $H$-linearity. Write $g=\sum_\nu b_\nu\otimes c_\nu$ with $b_\nu\in\A$ and $c_\nu\in H$. Expanding each of the finitely many exponentials gives
\begin{align*}
 \sum_\nu c_\nu\exp_\Phi\left(\frac{b_\nu(\theta)z}{a(\theta)}\right)
 &=\sum_{m\ge0}\frac{z^{q^m}}{D_m^\Phi}
 \frac{\sum_\nu c_\nu b_\nu(\theta^{q^m})}{a(\theta)^{q^m}}.
\end{align*}
Since $a\in\Fq(t)$, $a(\theta)^{q^m}=a(\theta^{q^m})$. This proves \eqref{Eq:EEefandds}. Convergence follows from the finite sum of entire exponentials.
\end{proof}

\begin{definition}[The generating function {\cite[Definitions~6.20--6.21]{HuHuangYau2026}}]\label{def:H-Phi-revised}
For $z\in\Cinf$, the generating function is
\begin{equation}\label{eq:H-Pellarin-series}
 H_\Phi(z;t)=\gamma_\zeta\sum_{m\ge0}\frac{z^{q^m}}{D_m^\Phi}
 \frac{\beta(\theta^{q^m})}{x(\theta^{q^m})-x(t)}.
\end{equation}
On domains of convergence we abbreviate this series by
\begin{align*}
 \gamma_\zeta\EE_{\Phi,z}\left(\frac{\beta(u)}{x_u-x_t}\right).
\end{align*}
\end{definition}

\begin{proposition}[Analytic domain and convergence]\label{prop:H-convergence}
For every $z\in\Cinf$, \eqref{eq:H-Pellarin-series} converges in $\mathbb T_\zeta$. It defines a meromorphic function on the rigid-analytic space
\begin{align}\label{eq:mero-domain}
 \left(\mathbf P^1_{\Cinf}\setminus\{\zeta^{(s)}:0\le s<N\}\right)^{\rm an},
\end{align}
obtained from the projective line by removing the $N$ punctures,
with at most simple poles at $\theta^{q^m}$.
\end{proposition}
\begin{proof}
Let $a_m=\theta^{q^m}$ and let
\begin{align*}
 k_m(t)=\frac{\beta(a_m)}{x(a_m)-x(t)}
 =\frac{(a_m-\zeta^{(1)})(t-\zeta)}{t-a_m}.
\end{align*}
The expression on the right is regular at $P_\infty$, with value $a_m-\zeta^{(1)}$. The distance from $a_m$ to $\zeta^{(m)}$ is $q^{-q^m}<1$, while $|a_m-\zeta^{(1)}|\le1$. On $\mathbb D_\zeta$, the ultrametric inequality gives
$|t-a_m|=|t-\zeta^{(m)}|\ge1$ and $|t-\zeta|\le|t-\zeta^{(m)}|$. Hence $\|k_m\|\le1$. Since the exponential is entire, $|z^{q^m}/D_m^\Phi|\to0$, and the series converges in the complete affinoid algebra.

On an affinoid subdomain of \eqref{eq:mero-domain} that
avoids the poles $\theta^{q^m}$, the finitely many initial terms
$k_0,\ldots,k_{m_0}$ are bounded. For $m>m_0$ each $a_m$ lies in a
small disc around one of the removed points $\zeta^{(s)}$, and the
rational functions $k_m$ are then uniformly bounded on the subdomain.
Near $P_\infty$ the same bound holds in the coordinate $t^{-1}$.
Together with $|z^{q^m}/D_m^\Phi|\to0$ this gives normal convergence
on \eqref{eq:mero-domain}. For each $m$, a sufficiently small disc around $a_m$ contains
no other $a_j$. On this disc the sum over $j\ne m$ is analytic, and
the $m$-th term has at most a simple pole. This proves the meromorphic assertion.
The accumulation $\theta^{q^{s+kN}}\to\zeta^{(s)}$ explains why the
punctures $\zeta^{(s)}$ are removed in \eqref{eq:mero-domain}.
\end{proof}

The proof of Proposition~\ref{prop:rank-one-inputs} below works in
$\mathbb T_\zeta$ and uses its Frobenius fixed ring. For this we
compare $\mathbb T_\zeta$ with the completed tensor product used in
\cite[Definitions~2.1 and~2.4]{GazdaMaurischat2021}. The comparison
identifies the two $\A$-actions, determines the fixed ring
$\mathbb T_\zeta^{\tau=1}$, and justifies the evaluation maps used
in Section~\ref{sec:gauss}.

Choose an $\Fq$-basis $(a_i)_{i\ge0}$ of $\A$. For a
finite sum in $\A\otimes_{\Fq}\Cinf$ put
\[
 \left\|\sum_i a_i\otimes c_i\right\|_{\mathrm G}
 =\max_i|c_i|_\rho.
\]
This maximum is a norm on $\A\otimes_{\Fq}\Cinf$, independent of
the chosen basis. $\A\widehat\otimes_{\Fq}\Cinf$ denotes its
completion.

\begin{lemma}[Comparison of the analytic algebras]
\label{lem:Tate-algebra-comparison}
 The map
\begin{equation}\label{eq:Tate-algebra-comparison}
 \jmath:\A\widehat\otimes_{\Fq}\Cinf
 \xrightarrow{\ \sim\ }\mathbb T_\zeta,
 \qquad a\otimes c\longmapsto c\,a(t),
\end{equation}
is an isometric $\Cinf$-algebra isomorphism. For every
$a\in\A$,
\begin{equation*}
 \jmath((ab)\otimes c)=a(t)\,\jmath(b\otimes c),
 \qquad
 \jmath\bigl(b\otimes\Psi_a(c)\bigr)=\Psi_a\bigl(\jmath(b\otimes c)\bigr),
\end{equation*}
where the second $\Psi_a$ acts on the coefficients of the function,
and
\(
 \jmath(b\otimes c^{q})=\tau\bigl(\jmath(b\otimes c)\bigr).
\)
\end{lemma}
\begin{proof}
A change of $\Fq$-basis has coefficients of absolute value at
most one. Applying the ultrametric inequality
$|x+y|_\rho\le\max\{|x|_\rho,|y|_\rho\}$ to the change matrix
and to its inverse proves that the displayed maximum is basis
independent. This is the Gauss norm of
\cite[Proposition~2.2]{GazdaMaurischat2021}.

Every $a\in\A$ can be written as $f(t)/\rho(t)^d$, with
$f\in\Fq[t]$ and $\deg f\le Nd$. At a finite point of
$\mathbb D_\zeta$ with $|t|_\rho\le1$, all the distances
$|t-\zeta^{(s)}|_\rho$ equal one, so $|a(t)|_\rho\le1$.
If $|t|_\rho>1$, the denominator has absolute value $|t|_\rho^{Nd}$
and the numerator has absolute value at most that number. At
$P_\infty$, the same bound follows by taking the value of $a(t)$ in
the coordinate $t^{-1}$. Thus $\|a\|_{\mathbb D_\zeta}\le1$,
and for $F=\sum_i a_i\otimes c_i$ the supremum norm of its image in $\mathbb T_\zeta$ is at most $\|F\|_{\mathrm G}$.

For the reverse inequality, let $F\ne0$ and pick a coefficient
$c_{i_0}$ of maximal absolute value. The scalars $c_i/c_{i_0}$ lie in
the valuation ring of $\Cinf$, and $c_{i_0}/c_{i_0}=1$ has nonzero
residue. Let $\widetilde c_i$ be the residue of $c_i/c_{i_0}$ in
$\overline{\Fq}$. The rational function
$\sum_i\widetilde c_i a_i(t)$ is nonzero. Indeed, linear independence
of the $a_i$ over $\Fq$ survives the constant extension to
$\overline{\Fq}$. To see this, clear a common power of $\rho$ and
extend scalars in the resulting finite-dimensional polynomial space.
Choose $\eta\in\overline{\Fq}$ away from the zeros of
$\sum_i\widetilde c_i a_i(t)$ and from the punctures. Viewed as a
constant of $\Cinf$, this point belongs to $\mathbb D_\zeta$.
Evaluation of $\jmath(F)(\eta)/c_{i_0}$ has nonzero residue
$\sum_i\widetilde c_i a_i(\eta)$, hence absolute value one. This
proves equality of the two norms.

The image of the finite tensor product $\A\otimes_{\Fq}\Cinf$ is
$\Cinf\otimes_H A_H$. It contains every
$Y_s=(t-\zeta^{(s)})^{-1}$ and all scalars in $\Cinf$. These functions
topologically generate $\mathbb T_\zeta$, and hence the image is dense. The isometry
therefore extends to the isomorphism~\eqref{eq:Tate-algebra-comparison}.

On finite tensors, multiplication by $a$ on
$\A\otimes_{\Fq}\Cinf$ becomes multiplication by $a(t)$ on
$\mathbb T_\zeta$, and $\operatorname{id}_{\A}\otimes\tau$
becomes coefficient Frobenius because $a(t)$ has coefficients in
$\Fq$. The norm identity $\|\tau F\|_{\mathrm G}=\|F\|_{\mathrm G}^{q}$
proves continuity of $\tau$. Consequently,
$\operatorname{id}_{\A}\otimes\Psi_a$ becomes the coefficientwise
polynomial $\Psi_a$ on $\mathbb T_\zeta$. These two identities
extend to the completion $\A\widehat\otimes_{\Fq}\Cinf$.
Every element of the completion has a unique expansion
$\sum_{i\ge0}a_i\otimes c_i$ with $|c_i|_\rho\to0$.
It is fixed by coefficient Frobenius precisely when $c_i^q=c_i$
for every $i$. Thus all $c_i$ lie in $\Fq$, and only finitely many
are nonzero. This proves $\mathbb T_\zeta^{\tau=1}=\A$.
Evaluation at $\eta$ is bounded by the supremum norm and sends
$a\otimes c$ to $a(\eta)c$ on finite tensors.
For the tensor-valued version, put $F_\eta=\Fq(\eta)$,
$\chi_\eta(a)=a(\eta)$, and write
$m\colon F_\eta\otimes_{\Fq}\Cinf\to\Cinf$ for the
multiplication map $c'\otimes z\mapsto c'z$. The map
\[
 \chi_\eta\otimes\operatorname{id}:
 \A\otimes_{\Fq}\Cinf\longrightarrow F_\eta\otimes_{\Fq}\Cinf,
 \qquad a\otimes c\longmapsto a(\eta)\otimes c,
\]
is bounded for the Gauss norms. Indeed, in an $\Fq$-basis of
$F_\eta$, every $\chi_\eta(a_i)$ has coefficients in $\Fq$, so the
output norm is at most $\|F\|_{\mathrm G}$. The target
$F_\eta\otimes_{\Fq}\Cinf$ is a finite direct sum of copies of
$\Cinf$ and is complete. Thus $\chi_\eta\otimes\operatorname{id}$
extends continuously, and its composite
$m\circ(\chi_\eta\otimes\operatorname{id})$ with the
multiplication map is evaluation at $\eta$.
\end{proof}

For a finite constant $\eta$ outside the punctures, evaluation is the
composite of the continuous map $\chi_\eta\otimes\operatorname{id}$
with scalar multiplication. In Section~\ref{sec:gauss}, the tensor
factors are retained until this last multiplication.

The following identities are stated in
\cite[Proposition~6.18, Theorem~6.19 and Theorem~6.26]{HuHuangYau2026},
with the same choice of $\kappa$ as in \eqref{eq:Phi-period-explicit}.
Put $\omega_\Phi=H_\Phi(\wt\pi_\Phi;\cdot)$.
We prove these identities in $\mathbb T_\zeta$, using the
fixed-ring assertion of Lemma~\ref{lem:Tate-algebra-comparison}.

\begin{proposition}[Frobenius and period identities]\label{prop:rank-one-inputs}
For every $z\in\Cinf$,
\begin{align*}
 \tau H_\Phi(z;t)&=f_\zeta(t)H_\Phi(z;t)+\exp_\Phi(z),
\end{align*}
and
\begin{equation}\label{eq:omega-H-equality}
 \omega_\Phi(t)
 =\kappa\prod_{m\ge0}
 \frac{\Theta^{q^m}}{\Theta^{q^m}-(t-\zeta^{(m)})^{-1}}.
\end{equation}
In particular,
\begin{equation}\label{eq:omega-special-equation}
 \tau\omega_\Phi=f_\zeta\omega_\Phi.
\end{equation}
\end{proposition}

\begin{proof}
We verify the identities in the affinoid algebra $\mathbb T_\zeta$.
Write $d_m=s_m(\theta^{q^m})$, with $d_0=1$.
The exponential formula for $\Psi$ is
$\exp_\Psi(z)=\sum_{m\ge0}z^{q^m}/d_m$
\cite[Theorem~3.14 and Corollary~3.15]{HuHuangYau2026}.
This formula can also be checked directly from the specified model.
Substituting the identities \(d_m=s_j(\theta^{q^m})d_{m-j}^{q^j}\) for \(0\le j\le m\),
together with \(s_j(\theta^{q^m})=0\) for \(j>m\), into the defining expansion of
\(a(t)\) gives, for every \(0\ne a\in\A\),
\begin{align*}
 \sum_{j=0}^{\min\{m,n(a)\}}
 \frac{b_j(a)}{d_{m-j}^{q^j}}
 =\frac{a(\theta^{q^m})}{d_m}
 =\frac{a(\theta)^{q^m}}{d_m}.
\end{align*}
This is the coefficient identity for
$\Psi_a\exp_\Psi(z)=\exp_\Psi(a(\theta)z)$.
Since $s_m=f_\zeta s_{m-1}^{(1)}$, we have
\begin{align*}
 d_m=d_{m-1}^{q}
 \frac{\theta-\theta^{q^m}}{(\theta^{q^m}-\zeta)(\theta-\zeta)}.
\end{align*}
Conjugation in
\eqref{eq:Phi-explicit-normalization} gives
$(D_m^\Phi)^{-1}=C^{q^m-1}/\sigma^2(d_m)$, where $\sigma$ is the
automorphism fixing $\theta$. Moreover,
\begin{align*}
 C^{q-1}
 &=\gamma_\zeta^q\frac{\theta-\zeta^q}{\theta-\zeta^{q^2}},&
 (D_m^\Phi)^{-1}
 &=\gamma_\zeta^q (D_{m-1}^\Phi)^{-q}
 \frac{(\theta^{q^m}-\zeta^{q^2})(\theta-\zeta^q)}{\theta-\theta^{q^m}}
 \quad(m\ge1).
\end{align*}
Set
\begin{align*}
 h_m(t)=\gamma_\zeta (D_m^\Phi)^{-1}
 \frac{(\theta^{q^m}-\zeta^q)(t-\zeta)}{t-\theta^{q^m}},
 \qquad H_\Phi(z;t)=\sum_{m\ge0}h_m(t)z^{q^m}.
\end{align*}
Direct calculation using this recurrence gives
\begin{align*}
 f_\zeta h_0&=-1,&
 f_\zeta h_m+(D_m^\Phi)^{-1}
 &=(D_m^\Phi)^{-1}\frac{\theta-\theta^{q^m}}{\theta-\zeta^q}
 \frac{t-\zeta^q}{t-\theta^{q^m}}
 =h_{m-1}^{(1)}\quad(m\ge1).
\end{align*}

For fixed $z\in\Cinf$, coefficient Frobenius also acts on the scalar $z$, so
$\tau H_\Phi(z;t)=\sum_{m\ge0}h_m^{(1)}(t)z^{q^{m+1}}$.
Summing proves $\tau H_\Phi=f_\zeta H_\Phi+\exp_\Phi(z)$.
Proposition~\ref{prop:H-convergence} justifies convergence, and
Lemma~\ref{lem:Tate-algebra-comparison} proves continuity of $\tau$.

Let $\Pi$ denote the product on the right of
\eqref{eq:omega-H-equality}. Since
$\|Y_m/\Theta^{q^m}\|\le q^{-q^m}$, the product of the units
$(1-Y_m/\Theta^{q^m})^{-1}$ converges to a unit of
$\mathbb T_\zeta$. Shifting indices and using
$\kappa^{q-1}=-\Theta$ gives
\begin{align*}
 \tau\Pi=-\Theta(1-x/\Theta)\Pi=f_\zeta\Pi.
\end{align*}
Both $H_\Phi(\wt\pi_\Phi;t)$ and $\Pi(t)$ satisfy the
homogeneous Frobenius equation, since $\exp_\Phi(\wt\pi_\Phi)=0$.
Their quotient is fixed by $\tau$. The fixed-ring assertion of
Lemma~\ref{lem:Tate-algebra-comparison} therefore gives
$H_\Phi(\wt\pi_\Phi;t)/\Pi(t)=a(t)\in\A$.
Each factor of the product is
$1+(\theta-\zeta)^{q^m}/(t-\theta^{q^m})$. Hence the same separation from
the punctures $\zeta^{(s)}$ used in Proposition~\ref{prop:H-convergence}
gives meromorphic continuation of $\Pi$ to the same domain.
The identity $H_\Phi(\wt\pi_\Phi;t)=a(t)\Pi(t)$ extends there
by the identity principle on this connected curve.
Only the $m=0$ summand, respectively factor, contributes a pole at
$t=\theta$. Using \eqref{eq:Phi-period-explicit}, their residues are
\begin{align*}
 \operatorname{Res}_{t=\theta}H_\Phi(\wt\pi_\Phi;t)\,dt
 &=\wt\pi_\Phi\gamma_\zeta
 (\theta-\zeta^q)(\theta-\zeta)
 =\wt\pi_\Phi(\theta-\zeta)^2,\\
 \operatorname{Res}_{t=\theta}\Pi(t)\,dt
 &=\kappa(\theta-\zeta)
 \prod_{m\ge1}\frac{\Theta^{q^m}}
 {\Theta^{q^m}-\sigma^m(\Theta)}
 =\wt\pi_\Phi(\theta-\zeta)^2.
\end{align*}
Thus $a(\theta)=1$. The generic characteristic embedding is
injective, so $a=1$, proving \eqref{eq:omega-H-equality} and hence
\eqref{eq:omega-special-equation}.
\end{proof}

\subsection{Remainder of the generating function}
\label{subsec:H-remainder}

Keep the element $p$ and the integer $n=Nd$ of Section~\ref{sec:pairing}.
Over $\Cinf$, write the complete zero divisor as
\[
 \operatorname{div}_0(p)
 =
 \sum_{\nu=1}^{h}m_\nu[\lambda_\nu]
 +m_\infty[P_\infty],
 \qquad
 m_\nu\in\mathbb Z_{>0},\quad m_\infty\in\mathbb Z_{\ge0}.
 \]
Here $\lambda_\nu$ are distinct finite constants outside the punctures
$P_s$. Equality of the degrees of the zero and
pole divisors gives $\sum_{\nu=1}^{h}m_\nu+m_\infty=n$.
Write $y_\nu=t-\lambda_\nu$ and $y_\infty=w=t^{-1}$.
Choose pairwise disjoint closed discs $\{|y_i|_\rho\le r_i\}$
with $r_i\in|\Cinf^\times|_\rho$ and $0<r_i<1$, sufficiently small that
$p=y_i^{m_i}u_i$ with $u_i$ a unit. These discs lie in $\mathbb D_\zeta$
and avoid every Frobenius pole $\theta^{q^m}$. The affinoid algebra on their disjoint union is
\begin{align*}
 \mathscr A_p
 =\prod_i\Cinf\!\left\langle\frac{y_i}{r_i}\right\rangle,
\end{align*}
where $i$ runs over $\{1,\ldots,h\}$, with $\infty$ adjoined when $m_\infty>0$.

\begin{proposition}\label{prop:analytic-quotient}
For every $z$, one has $H_\Phi(z;t)\in\mathscr A_p$ and
\begin{equation*}
 \mathscr A_p/p\mathscr A_p
 \simeq\prod_i\Cinf[y_i]/(y_i^{m_i}).
\end{equation*}
The factors include the jet algebra at $P_\infty$ when $m_\infty>0$,
and the quotient is independent of the admissible radii.
\end{proposition}
\begin{proof}
Proposition~\ref{prop:H-convergence} gives analyticity on every chosen disc.
Dividing by $p=y_i^{m_i}u_i$ gives the displayed jet algebra, since $u_i$
is a unit. Restriction to smaller discs induces the identity on each
jet algebra.
\end{proof}

Recall the representative space $V_p=\Span_H\{e_0,\ldots,e_{n-1}\}$.
\begin{proposition}\label{prop:Vp-analytic-representatives}
The natural map $\Cinf\otimes_H V_p\longrightarrow\mathscr A_p/p\mathscr A_p$ is an isomorphism.
\end{proposition}

\begin{proof}
Let $F\in\Cinf\otimes_H V_p$ map to zero in
$\mathscr A_p/p\mathscr A_p$. Then $F/p$ is regular at every zero of
$p$, including $P_\infty$ when $m_\infty>0$. At the punctures the pole
orders of $F$ are bounded by those of $V_p$. At $P_0$ the pole order of
$F$ is at most $d-1$ and that of $p$ is exactly $d$. At every other
$P_s$ the pole order of $F$ is at most $d$ and that of $p$ is exactly
$d$. Hence $F/p$ is regular at $P_0,\ldots,P_{N-1}$ as well, so it is
regular everywhere on $\mathbf P^1_{\Cinf}$ and is therefore a
constant. It vanishes at $P_0$, so $F=0$ and the map is injective.

The source has $\Cinf$-dimension $n$, since
$e_0,\ldots,e_{n-1}$ is an $H$-basis of $V_p$. The target has
$\Cinf$-dimension $n$ by the degree equality
$\sum_{\nu=1}^{h}m_\nu+m_\infty=n$. An injective map between spaces of
the same finite dimension is an isomorphism.
\end{proof}

The dual functions $\DD_i$ of \eqref{eq:D-i-imrn}
determine the operator coefficients used below. For $0\le i<n$ put
\begin{equation*}
 \mathscr D_{p,i}=\beta\,\DD_i/\alpha_n.
\end{equation*}

For $i=n-1$, \eqref{eq:D-i-closed-imrn} gives
$\DD_{n-1}=\alpha_n(1+c_{n-1}x)^{-1}$. Since $n=Nd$, periodicity gives
$c_{n-1}=c_{N-1}$, while $\beta=1+c_{N-1}x$. Dividing by
$\alpha_n$ gives $\mathscr D_{p,n-1}=1$, recorded as
\begin{equation}\label{eq:D-last-normalized}
 \mathscr D_{p,n-1}=1.
\end{equation}

\begin{definition}[The $p$-remainder]\label{def:p-remainder-imrn}
For $F\in\mathscr A_p$, its $p$-remainder is the unique element of $\Cinf\otimes_H V_p$ with the same image as $F$ in $\mathscr A_p/p\mathscr A_p$, written
\begin{align*}
 [F]_p
 =\sum_{i=0}^{n-1}C_i(F)e_i(t)
 \in\Cinf\otimes_H V_p.
\end{align*}
\end{definition}

At a finite zero, equality of remainders is equivalent to equality of
the first $m_\nu$ Hasse--Schmidt coefficients in $t-\lambda_\nu$. At
infinity, it is equivalent to equality of the first $m_\infty$
coefficients in $w=t^{-1}$.

For $z\in\Cinf$, define the remainder
\begin{equation}\label{eq:R-p-def}
 R_p(z;t)=\gamma_\zeta\EE_{\Phi,z}\left(\frac{\beta(u)}{p(u)}K_p(u,t)\right).
\end{equation}

For every $H$-algebra $S$, scalar extension of Proposition~\ref{prop:quotient-basis-imrn} gives an isomorphism $S\otimes_H V_p\xrightarrow{\sim}S\otimes_H B_p$. Its inverse is denoted by $\operatorname{rem}^{\mathrm{alg}}_{p,t}$.

The universal Cauchy identity uses a second copy of the coefficient ring.
Let $S_u$ be the copy of $A_H$ with $t$ replaced by an
independent variable $u$, localized at $p(u)$. In particular, $x_u$,
$\beta(u)$, and $p(u)$ are elements of $S_u$.

\begin{lemma}[Remainder of the Cauchy kernel]\label{lem:Cauchy-remainder-imrn}
In $S_u\otimes_H B_p$, the element $x_u-x_t$ is a unit and
\begin{align*}
 \operatorname{rem}^{\mathrm{alg}}_{p,t}
 \left(\frac1{x_u-x_t}\right)
 =\frac{K_p(u,t)}{p(u)}
 =\frac{\OO^{[2]}_{p,\zeta}(u,t)}{p(u)}.
\end{align*}
Multiplying by $\beta(u)$ gives
\begin{equation}\label{eq:beta-Cauchy-coefficients}
 \operatorname{rem}^{\mathrm{alg}}_{p,t}
 \left(\frac{\beta(u)}{x_u-x_t}\right)
 =\frac{\beta(u)}{p(u)}K_p(u,t)
 =\sum_{i=0}^{n-1}
 \frac{\alpha_n\mathscr D_{p,i}(u)}{p(u)}e_i(t).
\end{equation}
If $u$ is specialized to $u_0\in\Cinf$ outside the punctures, with
$p(u_0)\ne0$, these formulas agree with the analytic $p$-remainders
whenever $x(u_0)-x(t)$ is a unit of $\mathscr A_p$.
\end{lemma}

\begin{proof}
By the definition of $K_p$, in $S_u\otimes_H B_p$ one has
\[
 (x_u-x_t)K_p(u,t)=p(u).
\]
The right side is invertible in $S_u$, so $K_p(u,t)/p(u)$ is the
inverse of $x_u-x_t$. Formula~\eqref{eq:K-dual-exp-imrn} puts this
inverse in $S_u\otimes_H V_p$, which is the selected representative
space. This proves the first assertion. Multiplication by $\beta(u)$
and the expansion~\eqref{eq:K-dual-exp-imrn} prove
\eqref{eq:beta-Cauchy-coefficients}.

Before reduction, the corresponding identity of rational functions is
\begin{equation}\label{eq:exact-kernel-division}
 \frac{\beta(u)}{x_u-x_t}
 -\frac{\beta(u)}{p(u)}K_p(u,t)
 =p(t)\frac{\beta(u)}{p(u)(x_u-x_t)}.
\end{equation}
Under the stated numerical specialization, the right side of
\eqref{eq:exact-kernel-division} is $p(t)$ times an element of
$\mathscr A_p$, and its middle term belongs to $\Cinf\otimes_H V_p$.
Here $x(u_0)-x(t)$ is a unit of $\mathscr A_p$ exactly so that the
Cauchy kernel has no pole on any chosen disc. Uniqueness in
Definition~\ref{def:p-remainder-imrn} therefore identifies the
algebraic and analytic remainders.
\end{proof}

\begin{theorem}[Exact division and remainder]
\label{thm:H-remainder-imrn}
For every $z\in\Cinf$, the remainder $R_p(z;t)$ of
\eqref{eq:R-p-def} lies in $\Cinf\otimes_H V_p$, and in $\mathscr A_p$
one has the exact identity
\begin{equation}\label{eq:exact-analytic-division}
 H_\Phi(z;t)-R_p(z;t)
 =p(t)H_\Phi\left(\frac{z}{p(\theta)};t\right),\qquad
 [H_\Phi(z;t)]_p=R_p(z;t).
\end{equation}
The coefficient formula is
\begin{equation}\label{eq:H-remainder-coeff-imrn}
 [H_\Phi(z;t)]_p
 =\gamma_\zeta\sum_{i=0}^{n-1}
 \alpha_n(\mathscr D_{p,i}\diamond_{\Phi,H}\exp_\Phi(z/p(\theta)))e_i(t).
\end{equation}
The last coefficient in the basis $(e_i)$ is
\begin{equation*}
 C_{n-1}(H_\Phi(z;\cdot))
 =\alpha_n\gamma_\zeta\exp_\Phi(z/p(\theta)).
\end{equation*}
\end{theorem}

\begin{proof}
The finite expansion \eqref{eq:beta-Cauchy-coefficients} and
Lemma~\ref{lem:exponential-action-series} give the convergent series
\begin{align*}
 R_p(z;t)=\gamma_\zeta\sum_{m\ge0}\frac{z^{q^m}}{D_m^\Phi}
 \frac{\beta(\theta^{q^m})}{p(\theta^{q^m})}K_p(\theta^{q^m},t),
\end{align*}
in the finite dimensional space $\Cinf\otimes_H V_p$.
Generic characteristic ensures $p(\theta^{q^m})\ne0$, and $p\in\Fq(t)$ gives
$p(\theta^{q^m})=p(\theta)^{q^m}$. Evaluate \eqref{eq:exact-kernel-division}
at $u=\theta^{q^m}$, with the input $z$ still fixed, multiply by $\gamma_\zeta z^{q^m}/D_m^\Phi$, and sum.
The other two series are those of $H_\Phi(z;t)$ and
$p(t)H_\Phi(z/p(\theta);t)$. Proposition~\ref{prop:H-convergence}
proves convergence in $\mathscr A_p$. This proves the exact identity
\eqref{eq:exact-analytic-division}, and Proposition~\ref{prop:Vp-analytic-representatives}
identifies $R_p$ with the unique remainder.

Finally,
$\EE_{\Phi,z}(\mathscr D_{p,i}/p)=\mathscr D_{p,i}\diamond_{\Phi,H}\exp_\Phi(z/p(\theta))$
gives \eqref{eq:H-remainder-coeff-imrn}. The last coefficient follows
from \eqref{eq:D-last-normalized}.
The identity holds on every chosen disc, with its full multiplicity.
\end{proof}

\begin{remark}
The proof uses only divisibility by $p(t)$. If a zero has
multiplicity $m_\nu$, it simultaneously proves equality of the first
$m_\nu$ Hasse--Schmidt coefficients at that zero.
\end{remark}

\subsection{Prime torsion}
\label{subsec:torsion-special}

From now on in this subsection, let $\pideal=(p)\subset\A$ be a nonzero
principal prime ideal, different from the characteristic ideal of
$\Phi$. Then $\A/\pideal\cong\mathbb F_{q^n}$. We keep the
principal period normalization $\ker(\exp_\Phi)=\wt\pi_\Phi\iota(\A)$
from Proposition~\ref{prop:Psi-Phi-comparison}, and define
\begin{align*}
 \mu_p=
 \exp_\Phi\!\left(\frac{\wt\pi_\Phi}{p(\theta)}\right).
\end{align*}
The $p$-remainder of $\omega_\Phi$ is
\begin{align*}
 [\omega_\Phi]_p
 =
 \sum_{i=0}^{n-1}C_i e_i.
\end{align*}
Its last coefficient is denoted by $L_p(\omega_\Phi)=C_{n-1}$.

Fix a subfield $\KPhi\subset\Cinf$ containing $\iota(K)$ and the coefficients of all $\Phi_a$.

\begin{proposition}\label{prop:rank-one-p-torsion}
The map
\begin{equation}\label{eq:rank-one-p-torsion-map}
 \A/\pideal\longrightarrow\Phi[\pideal],
 \qquad
 \bar a\longmapsto\Phi_a(\mu_p),
\end{equation}
is an isomorphism of $\A/\pideal$-modules. In particular,
$\mu_p\neq0$ and it generates $\Phi[\pideal]$. Also,
\begin{equation}\label{eq:leading-coeff-torsion}
 L_p(\omega_\Phi)=\alpha_n\gamma_\zeta\mu_p,
\end{equation}
and consequently
\begin{equation}\label{eq:coefficient-torsion-field}
 \KPhi(\zeta)(C_0,\dots,C_{n-1})
 =
 \KPhi(\zeta)(\Phi[\pideal]).
\end{equation}
\end{proposition}

\begin{proof}
Since $\Phi_p(\mu_p)=\exp_\Phi(\wt\pi_\Phi)=0$, the map
\eqref{eq:rank-one-p-torsion-map} is well defined. If
$\Phi_a(\mu_p)=0$, the functional equation gives
$\exp_\Phi(a(\theta)\wt\pi_\Phi/p(\theta))=0$, so
$a(\theta)\wt\pi_\Phi/p(\theta)=b(\theta)\wt\pi_\Phi$ for some
$b\in\A$. As $\wt\pi_\Phi\neq0$ and $\A\hookrightarrow\Cinf$ is
injective, $a=pb$ and $\bar a=0$. Both sides of
\eqref{eq:rank-one-p-torsion-map} have $q^n$ elements (the target
because $\Phi_p$ is separable of degree $q^n$), so the map is an
isomorphism.

Theorem~\ref{thm:H-remainder-imrn} at $z=\wt\pi_\Phi$ gives
\eqref{eq:leading-coeff-torsion}. Each $C_i$ is
$\alpha_n\gamma_\zeta$ times an $H$-linear combination of $\Phi_a(\mu_p)$.
Since
$\alpha_n,\gamma_\zeta\in\KPhi(\zeta)^\times$, all coefficients lie in
$\KPhi(\zeta)(\Phi[\pideal])$, while the last coefficient recovers
$\mu_p$. This proves \eqref{eq:coefficient-torsion-field}.
\end{proof}

\begin{remark}\label{rem:split-coeff-not-torsion}
The individual coefficients $C_i$ need not be torsion points.
The operators $\mathscr D_{p,i}$ may contain scalars in $H$ that do
not commute with the Frobenius terms in a Drinfeld polynomial.
Proposition~\ref{prop:rank-one-p-torsion} is a statement about the
\emph{field} generated by the coefficients. The leading coefficient
is a nonzero scalar multiple of the torsion generator
$\mu_p$.
\end{remark}

\subsection{One value for the torsion extension}
\label{subsec:one-value}

The prime $\pideal$ corresponds to a closed point $Q$ of
$\mathbf P^1_{\Fq}\setminus\{P_\rho\}$ with residue field
$\kappa(Q)=\A/\pideal$. A zero $\lambda\in\overline{\Fq}$ of $p$
identifies $\kappa(Q)$ with $\Fq(\lambda)=\mathbb F_{q^n}$, and the
$n$ zeros of $p$ are the distinct conjugates
$\lambda_s=\lambda^{q^s}$ $(0\le s<n)$. The principal divisor of
$p$ has pole divisor $dP_\rho$ and zero divisor $Q$, so $n=\deg Q=dN$.
Hence
\begin{equation}\label{eq:lambda-field}
 H=\mathbb F_{q^N}\subset\Fq(\lambda).
\end{equation}

Recalling $e_i=T_i^{(-1)}$, define the evaluation matrix
\(
 \mathcal V_{p,\lambda}
 =
 \bigl(e_i(\lambda_s)\bigr)_{0\le s,i<n}
\)
and the weights
\begin{equation}\label{eq:interpolation-residue-weights}
 b_s=\frac{\alpha_n}{\beta(\lambda_s)\partial_xp(\lambda_s)}
 \qquad(0\le s<n).
\end{equation}
The residue pairing at the zeros identifies the inverse
of this matrix.

\begin{lemma}\label{lem:evaluation-matrix}
The matrix $\mathcal V_{p,\lambda}$ is invertible over
$\KPhi(\lambda)$, and its inverse has last row $(b_0,\ldots,b_{n-1})$.
\end{lemma}

\begin{proof}
The zeros $\lambda_0,\dots,\lambda_{n-1}$ are distinct, so $p$ is separable. By
the Chinese remainder theorem,
\begin{equation}\label{eq:CRT-evaluation}
 B_p\tensor_H\KPhi(\lambda)
 \longrightarrow
 \prod_{s=0}^{n-1}\KPhi(\lambda),
 \qquad
 \bar F\longmapsto
 \bigl(F(\lambda_s)\bigr)_{s=0}^{n-1}
\end{equation}
is an isomorphism. Since $H\subset\KPhi(\lambda)$ by
\eqref{eq:lambda-field}, the matrix of
\eqref{eq:CRT-evaluation} in the basis $e_0,\dots,e_{n-1}$ is
exactly $\mathcal V_{p,\lambda}$.

For $F\in\KPhi(\lambda)\otimes_H V_p$, the global residue theorem and $\omega_p=-dx/p$ give
\begin{align*}
 \ell_p(F)=\sum_{s=0}^{n-1}\frac{F(\lambda_s)}{\partial_xp(\lambda_s)}.
\end{align*}
By the duality in Proposition~\ref{thm:dual-Casimir-imrn}, the last coefficient of $F$ is $\ell_p(\DD_{n-1}F)$. Formula~\eqref{eq:D-i-imrn} gives $\DD_{n-1}=\alpha_n/\beta$. Hence
\begin{align*}
 C_{n-1}(F)=\sum_{s=0}^{n-1}
 \frac{\alpha_nF(\lambda_s)}{\beta(\lambda_s)\partial_xp(\lambda_s)}.
\end{align*}
Taking $F=e_i$ identifies the weights with the last row of the inverse matrix. The denominators are nonzero because the zeros are simple and avoid the punctures $P_s$.
\end{proof}

On the generators $Y_m$ one has $(\tau Y_m)(t^q)=Y_m(t)^q$, and
continuity extends this to every element of $\mathbb T_\zeta$.
Applied to $\tau\omega_\Phi=f_\zeta\omega_\Phi$, this gives
\begin{equation}\label{eq:omega-total-q}
 \omega_\Phi(t)^q
 =
 f_\zeta(t^q)\omega_\Phi(t^q)
 =
 \left(\frac{1}{t^q-\zeta}-\frac{1}{\theta-\zeta}\right)
 \omega_\Phi(t^q).
\end{equation}

Put $w_s=\omega_\Phi(\lambda_s)$ for $0\le s<n$. Evaluating
\eqref{eq:omega-total-q} at $\lambda_s$ gives the recursion
\begin{equation}\label{eq:w-recursion}
 w_{s+1}=\frac{w_s^q}{f_\zeta(\lambda_{s+1})},
\end{equation}
whose denominators are nonzero ($f_\zeta$ vanishes only at
the transcendental point $\theta$, while every $\lambda_s$ is a
constant). Writing them as
\begin{align*}
 \delta_0=1,
 \qquad
 \delta_s=\prod_{j=1}^{s}f_\zeta(\lambda_j)^{q^{s-j}}
 \quad(1\le s<n),
\end{align*}
induction in \eqref{eq:w-recursion} yields
\begin{equation}\label{eq:w-from-w0}
 w_s=\frac{w_0^{q^s}}{\delta_s}.
\end{equation}

\begin{proposition}\label{thm:one-value-torsion}
For every finite zero $\lambda\in\overline{\Fq}$ of $p$, one value of
$\omega_\Phi$ generates the full $\pideal$-torsion extension:
\begin{equation}\label{eq:one-value-torsion}
 \KPhi(\lambda,\omega_\Phi(\lambda))
 =
 \KPhi(\lambda,L_p(\omega_\Phi))
 =
 \KPhi(\lambda,\Phi[\pideal]).
\end{equation}
The last coefficient $L_p(\omega_\Phi)$ is recovered from
$\omega_\Phi(\lambda)$ by 
\begin{equation}\label{eq:q-linear-leading-recovery}
 L_p(\omega_\Phi)
 =
 \sum_{s=0}^{n-1}b_sw_s
 =
 \sum_{s=0}^{n-1}
 \frac{b_s}{\delta_s}\,\omega_\Phi(\lambda)^{q^s},
\end{equation}
and the torsion generator by
\begin{equation}\label{eq:explicit-torsion-recovery}
 \mu_p=\gamma_\zeta^{-1}\sum_{s=0}^{n-1}
 \frac{\omega_\Phi(\lambda)^{q^s}}
 {\beta(\lambda_s)\partial_xp(\lambda_s)\delta_s}.
\end{equation}
\end{proposition}

\begin{proof}
Since $\omega_\Phi-[\omega_\Phi]_p$ is divisible by $p$, evaluating
at the $n$ zeros gives
\begin{equation}\label{eq:evaluation-system}
 \begin{pmatrix} w_0\\ \vdots\\ w_{n-1}\end{pmatrix}
 =
 \mathcal V_{p,\lambda}
 \begin{pmatrix} C_0\\ \vdots\\ C_{n-1}\end{pmatrix}.
\end{equation}
Put $L=\KPhi(\lambda)$. Equation \eqref{eq:lambda-field} gives $H\subset L$,
so every $b_s$ and $\delta_s$ lies in $L$. The invertible evaluation
matrix, Proposition~\ref{prop:rank-one-p-torsion}, and
\eqref{eq:w-from-w0} give, respectively,
\begin{align*}
 L(w_0,\ldots,w_{n-1})
 =L(C_0,\ldots,C_{n-1})
 =L(\Phi[\pideal])
 =L(w_0).
\end{align*}
The last equality uses only forward $q$-powers and the nonzero
$\delta_s$. Taking the last row of the inverse of
\eqref{eq:evaluation-system} and substituting \eqref{eq:w-from-w0}
proves \eqref{eq:q-linear-leading-recovery}. Substituting the weights
\eqref{eq:interpolation-residue-weights} and
$L_p(\omega_\Phi)=\alpha_n\gamma_\zeta\mu_p$ gives
\eqref{eq:explicit-torsion-recovery}. Finally,
$\alpha_n,\gamma_\zeta\in L^\times$ and \eqref{eq:leading-coeff-torsion}
show that the last coefficient generates $L(\mu_p)=L(\Phi[\pideal])$.
This proves all asserted field equalities.
\end{proof}

\begin{remark}[The ordinary point at infinity]\label{rem:one-value-finite-only}
The proof of \eqref{eq:one-value-torsion} uses the coordinate $t$, but the statement also holds at $P_\infty$ in the only possible case. If the prime ideal of $P_\infty$ is principal, say $(p)$, then $\operatorname{div}(p)=P_\infty-dP_\rho$ gives $dN=1$, hence $N=d=1$. Then $\A=\Fq[x]$ and $\Phi_x=\Theta+\tau$ is the Carlitz module. With $v=\omega_\Phi(P_\infty)=\kappa\ne0$ one has $\Phi_x(v)=0$, so $\Phi[(x)]=\Fq v$ and $\KPhi(v)=\KPhi(\Phi[(x)])$.
\end{remark}

\subsection{A complete example with a quadratic puncture}
\label{subsec:worked-example}

We compute the operator, the remainder and the torsion
recovery formula for a quadratic puncture. Take $q=3$, $\rho(t)=t^2+1$, and choose $\zeta\in H=\mathbb F_9$
with $\zeta^2=-1$. Thus $\zeta^3=-\zeta$ and $N=2$.
Throughout this example, $\Phi$ is the specified normalized rank-one
module of Subsection~\ref{subsec:normalized-dual}, with the period
normalization $\ker\exp_\Phi=\wt\pi_\Phi\iota(\A)$.
The generators of $\A$ and the chosen prime are
\begin{align*}
 a_0(t)=\frac1{t^2+1},\qquad
 a_1(t)=\frac{t}{t^2+1},\qquad
 p(t)=\frac{t^2+t+2}{t^2+1}=1+a_0(t)+a_1(t).
\end{align*}
Here $\A=\mathbb F_3[a_0,a_1]$, with
$a_1^2=a_0-a_0^2$. The numerator of $p$ is irreducible over
$\mathbb F_3$, is coprime to $\rho$, and $p(P_\infty)=1$.
Consequently $(p)$ is a finite principal prime, its residue field has
nine elements, and $d=1$, $n=2$. The two zeros of $p$ are
\begin{align*}
 \lambda_0=\lambda=1+\zeta,\qquad
 \lambda_1=\lambda^3=1-\zeta.
\end{align*}

For brevity set $\alpha=1-\zeta$, $b=1+\zeta$, and
$e(t)=(t+\zeta)^{-1}$. The expansion in $(T_k)$ and the representative
basis are
\begin{align*}
 T_0=1,\qquad T_1=\frac1{t-\zeta},\qquad
 T_2=\frac1{t^2+1},\qquad
 p=T_0+T_1+\alpha T_2,\qquad
 (e_0,e_1)=(1,e).
\end{align*}
In particular $\alpha_2=\alpha\ne0$ and
$V_p=\Span_H\{1,e\}$. The residue dual functions and $K_p$ are
\begin{align*}
 \DD_0(t)&=1+\alpha e(t)=\frac{t+1}{t+\zeta},&
 \DD_1(t)&=\alpha+b e(t)
 =\alpha\frac{t-\zeta}{t+\zeta},\\
 K_p(u,v)&=\DD_0(u)+\DD_1(u)e(v)\\
 &=1+\alpha\bigl(e(u)+e(v)\bigr)+b e(u)e(v)
 =\frac{uv+u+v-1}{(u+\zeta)(v+\zeta)}.
\end{align*}
For example, direct residues for $\omega_p=-dx/p$ give
$\ell_p(1)=b$, $\ell_p(e)=-\alpha$, and $\ell_p(e^2)=1$.
Since $\alpha^2=\zeta$ and $\alpha b=-1$, these values verify
$\langle\DD_j\omega_p,e_i\rangle_H=\delta_{ij}$.
The relation in the quotient is
\begin{align*}
 e(t)^2\equiv b-\zeta e(t)\pmod p.
\end{align*}
It also makes the higher representative explicit. Writing
$H_j=e(t_j)$, multiplying $K_p(t_1,t_3)$ and $K_p(t_2,t_3)$ and then reducing gives
\begin{align*}
 \OO_{p,\zeta}^{[3]}(t_1,t_2,t_3)
 &=\operatorname{rem}_{p,t_3}
 \bigl(K_p(t_1,t_3)K_p(t_2,t_3)\bigr)\\
 &=\zeta+\zeta(H_1+H_2+H_3)
 +H_1H_2+H_1H_3+H_2H_3.
\end{align*}

The analytic normalizing factors are
\begin{align*}
 \Theta=\frac1{\theta-\zeta},\qquad
 \beta(t)=\frac{t+\zeta}{t-\zeta},\qquad
 \gamma_\zeta=\frac{\theta-\zeta}{\theta+\zeta}.
\end{align*}
The normalization of $\Phi$ gives the additive
polynomials on the generators of $\A$,
\begin{align*}
 \Phi_{a_0}(V)
 &=\frac{V}{\theta^2+1}
 +\frac{\theta}{(\theta+\zeta)^2}V^3
 +\gamma_\zeta^9 V^9,\\
 \Phi_{a_1}(V)
 &=\frac{\theta}{\theta^2+1}V
 +\frac{\theta^2-1}{(\theta+\zeta)^2}V^3
 -\zeta\gamma_\zeta^9 V^9.
\end{align*}
Their coefficients follow by comparing the terms of degree $1,3,9$
in $\log_\Phi(\Phi_{a_j}(V))=a_j(\theta)\log_\Phi(V)$.
They commute under composition and satisfy
$\Phi_{a_1}\circ\Phi_{a_1}
 =\Phi_{a_0}-\Phi_{a_0}\circ\Phi_{a_0}$, as required by the relation
between $a_0$ and $a_1$. In particular,
\begin{align*}
 \Phi_p(V)
 =\frac{\theta^2+\theta+2}{\theta^2+1}V
 +\frac{\theta^2+\theta-1}{(\theta+\zeta)^2}V^3
 +\alpha\gamma_\zeta^9V^9.
\end{align*}
The operator coefficients are
\begin{align*}
 \mathscr D_{p,0}(t)
 =\frac{t+1}{\alpha(t-\zeta)},\qquad
 \mathscr D_{p,1}(t)=1.
\end{align*}
For any $z\in\Cinf$, put
$\mu_z=\exp_\Phi(z/p(\theta))$. Since
$(t-\zeta)^{-1}=a_1(t)+\zeta a_0(t)$, the coefficient formula becomes
\begin{align*}
 R_p(z;t)&=C_0(z)+C_1(z)e(t),\\
 C_0(z)&=\gamma_\zeta\bigl(
 \mu_z+b\Phi_{a_1}(\mu_z)
 +(\zeta-1)\Phi_{a_0}(\mu_z)\bigr),\\
 C_1(z)&=\alpha\gamma_\zeta\mu_z.
\end{align*}
The exact division identity, on the analytic neighbourhood used in
Theorem~\ref{thm:H-remainder-imrn}, is therefore
\begin{align*}
 H_\Phi(z;t)-C_0(z)-\frac{C_1(z)}{t+\zeta}
 =\frac{t^2+t+2}{t^2+1}\,
 H_\Phi\!\left(\frac{z}{p(\theta)};t\right).
\end{align*}

Now take $z=\wt\pi_\Phi$ and write $\mu_p=\mu_z$,
$C_i=C_i(z)$, and $w_s=\omega_\Phi(\lambda_s)$.
Since $\Phi_p(\mu_p)=0$, the preceding coefficients simplify to
\begin{align*}
 C_0=\gamma_\zeta\bigl(\Phi_{a_0}(\mu_p)-\zeta\mu_p\bigr),
 \qquad C_1=\alpha\gamma_\zeta\mu_p.
\end{align*}
Set $a=\alpha^{-1}=-1-\zeta$. Direct evaluation gives
\begin{align*}
 \mathcal V_{p,\lambda}
 &=\begin{pmatrix}1&a\\1&1\end{pmatrix},\qquad
 \mathcal V_{p,\lambda}^{-1}
 =\frac1\alpha\begin{pmatrix}-1&a\\1&-1\end{pmatrix},\\
 \begin{pmatrix}C_0\\C_1\end{pmatrix}
 &=\frac1\alpha
 \begin{pmatrix}-w_0+a w_1\\w_0-w_1\end{pmatrix}.
\end{align*}
Indeed $a-1=\alpha$, so the displayed matrix is invertible.
The Frobenius equation gives
\begin{align*}
 w_1=\frac{w_0^3}{f_\zeta(\lambda_1)},\qquad
 f_\zeta(\lambda_1)=\zeta-1-\Theta\ne0.
\end{align*}
Combining this with $\alpha^2=\zeta$ recovers the torsion generator
by the explicit additive polynomial
\begin{equation}\label{eq:example-torsion-recovery}
 \mu_p=
 \frac1{\zeta\gamma_\zeta}
 \left(w_0-\frac{w_0^3}{\zeta-1-\Theta}\right).
\end{equation}
Conversely, evaluation of the remainder gives
\begin{align*}
 w_0=\gamma_\zeta\bigl(\Phi_{a_0}(\mu_p)+\alpha\mu_p\bigr).
\end{align*}
The class of $a_0$ at $\lambda$ is $a\notin\mathbb F_3$.
Hence $1,\bar a_0$ form an $\mathbb F_3$-basis of $\A/(p)$, and
\begin{align*}
 \Phi[(p)]
 =\{i\mu_p+j\Phi_{a_0}(\mu_p):i,j\in\mathbb F_3\}.
\end{align*}
Finally $\zeta=\lambda-1$, so all coefficients in
\eqref{eq:example-torsion-recovery} lie in $\KPhi(\lambda)$.
The explicit form of both field equalities is
\begin{align*}
 \KPhi\bigl(\lambda,\omega_\Phi(\lambda)\bigr)
 =\KPhi(\lambda,\mu_p)
 =\KPhi\bigl(\lambda,\Phi[(p)]\bigr).
\end{align*}
Here $d=1$ and $\Fq(\lambda)=H$. The coefficients
$\gamma_\zeta^9$ and $-\zeta\gamma_\zeta^9$ of
$\Phi_{a_0}$ and $\Phi_{a_1}$ belong to $\KPhi$. Their ratio gives
$\zeta\in\KPhi$, hence $\lambda\in\KPhi$. Thus $\lambda$ may be
omitted from the three fields in this example.

\section{Gauss--Thakur sums at the chosen branch}
\label{sec:gauss}

We apply the evaluation theorem of Gazda--Maurischat
\cite[Theorem~5.2]{GazdaMaurischat2021} to the special function
$\omega_\Phi$ of $\Psi$. Fix $\pideal=(p)$
and a zero $\lambda$ of $p$. Then $\A/\pideal\simeq\mathbb F_{q^n}$ with $n=Nd$.
To distinguish the $q^N$-orbit used here from the $q$-orbit of
Section~\ref{sec:drinfeld}, write
$\lambda_{[j]}=\lambda^{q^{Nj}}$ for $0\le j<d$.

\subsection{The different and evaluation characters}
\label{subsec:split-different}

\begin{definition}[Evaluation character]
\label{def:Teichmuller-character}
The character determined by evaluation at $\lambda$ is
\begin{align*}
 \vartheta_{\pideal,\lambda}:(\A/\pideal)^\times
 \longrightarrow\mathbb F_{q^n}^{\times},
 \qquad \bar a\longmapsto a(\lambda).
\end{align*}
For $0\le j<d$, write $\chi$ for the zero-extension of
$\vartheta_{\pideal,\lambda}^{q^{Nj}}$, so that
$\chi(\bar a)=a(\lambda_{[j]})$.
\end{definition}

Since $N\mid n$, the algebra $(\A/\pideal)\otimes_{\Fq}H$ is a product of
$N$ copies of $\mathbb F_{q^n}$, each of degree $d$ over $H$.
The maps $a\otimes c\mapsto a(\lambda_{[j]})c$ are the $d$
$H$-embeddings of the factor selected by $\lambda$.

\begin{lemma}[The different]\label{lem:split-different}
The K\"ahler different agrees with the trace different used by
Gazda--Maurischat, and
\begin{equation}\label{eq:split-different}
 \dfrak_{\A/\Fq[p]}\otimes_{\Fq}H
 =\dfrak_{A_H/H[p]}=(\partial_xp).
\end{equation}
Here
\begin{equation}\label{eq:px-pt}
 \partial_xp=-(t-\zeta)^2\frac{dp}{dt},
\end{equation}
and $\partial_xp(\lambda_{[j]})\ne0$ for $0\le j<d$.
\end{lemma}

\begin{proof}
Since $(p)$ is prime, the zeros of $p$ are simple and $dp\ne0$.
Write $p=f(x)/h(x)$ with coprime $f,h\in H[x]$, choosing $f$ monic. Its pole divisor
$dP_\rho$ gives, up to nonzero constant factors,
\begin{align*}
 h(x)=\prod_{s=1}^{N-1}(1+c_sx)^d,
 \qquad \deg h=(N-1)d,\qquad \deg f=Nd.
\end{align*}
If $af+bh=1$, then $h(ap+b)=1$. Hence
$A_H=H[x,h^{-1}]=H[p,x]$.
The polynomial $F(X)=f(X)-ph(X)$ is monic, vanishes at $x$, and has degree $Nd=[H(x):H(p)]$. It is therefore the minimal polynomial of $x$. Thus
\[
 A_H\simeq H[p][X]/(F(X)).
\]
The different of the finite separable monogenic extension $A_H/H[p]$
is generated by
\[
 F'(x)=f'(x)-ph'(x)=h(x)\partial_xp.
\]
Since $h$ is a unit, this ideal is $(\partial_xp)$. The same ideal is the annihilator of
\[
 \Omega^1_{A_H/H[p]}=A_H\,dx/(\partial_xp)A_H\,dx.
\]

It remains to descend from $H$ to $\Fq$. The inverse different of $\A$ over $\Fq[p]$ commutes with the constant extension to $H$, by compatibility of the trace with flat base change. The same holds for the trace different of \cite[Section~3.2]{GazdaMaurischat2021}. For the monogenic extension above, the trace different is also generated by $F'(x)$, so both differents equal $(\partial_xp)$ over $H$. The module $\Omega^1_{\A/\Fq[p]}$ and its annihilator also commute with flat base change, so faithful flatness identifies the two differents over $\Fq$. This proves the equality of ideals in \eqref{eq:split-different}.

Finally, $dx=-(t-\zeta)^{-2}dt$ gives \eqref{eq:px-pt}. At each zero $\lambda_{[j]}$ the difference $x-x(\lambda_{[j]})$ is a local parameter, so $\partial_xp(\lambda_{[j]})\ne0$.
\end{proof}

\subsection{Special functions and the additive character}
\label{subsec:additive-character}

Lemma~\ref{lem:Tate-algebra-comparison} identifies $\mathbb T_\zeta$
with the completion used in \cite[Definition~3.1]{GazdaMaurischat2021},
including the two $\A$-actions and the evaluation maps. We use this
identification for both the period map and the evaluation theorem.

Put $\Lambda_\Psi=\ker\exp_\Psi$. The special functions of $\Psi$ are
\begin{align*}
 \operatorname{sf}(\Psi)=\{\omega\in\mathbb T_\zeta
 \mid
 a(t)\omega=\Psi_a(\omega)\text{ for all }a\in\A\}.
\end{align*}
Here $\Psi_a$ acts coefficientwise, fixing the function variable.
This is an $\A$-module, not generally an $H$-vector space.
Equation~\eqref{eq:omega-special-equation} gives
$\tau^j\omega_\Phi=s_j\omega_\Phi$. Therefore
\[
 \Psi_a(\omega_\Phi)
 =\sum_j b_j(a)\tau^j\omega_\Phi
 =\left(\sum_jb_j(a)s_j\right)\omega_\Phi
 =a(t)\omega_\Phi,
\]
so $\omega_\Phi\in\operatorname{sf}(\Psi)$.

\begin{proposition}[Gazda--Maurischat]
\label{prop:GM-period-recalled}
For the separating element $p$, there is an $\A$-linear isomorphism
\begin{equation}\label{eq:delta-p-GM}
 \delta_p:\dfrak_{\A/\Fq[p]}\Lambda_\Psi
 \xrightarrow{\ \sim\ }\operatorname{sf}(\Psi).
\end{equation}
\end{proposition}
This is \cite[Theorem~3.11]{GazdaMaurischat2021}. In its construction,
the balanced submodule of $\A\otimes_{\Fq[p]}\Lambda_\Psi$
consists of elements on which the two $\A$-actions agree.
Multiplication identifies it with
$\dfrak_{\A/\Fq[p]}\Lambda_\Psi$
\cite[Lemma~3.8]{GazdaMaurischat2021}.
For a lift to $\A\otimes_{\Fq}\Lambda_\Psi$, apply
$(\operatorname{id}\widehat\otimes\exp_\Psi)
(1\otimes p(\theta)-p\otimes1)^{-1}$ in the completed ambient
space. The result is independent of the lift.

Constant extension by $H$ gives the formal tensor
products $H\otimes_{\Fq}\Lambda_\Psi$ and
$\operatorname{id}_H\otimes\delta_p$, with action
$(a\otimes c)(d\otimes\ell)=cd\otimes a(\theta)\ell$. Multiplying
constants into $\Cinf$ need not commute with $\Psi_a$, so these
products stay formal.
The image of multiplication by $\partial_xp$ on
$H\otimes_{\Fq}\Lambda_\Psi$ is
$(\dfrak_{\A/\Fq[p]}\Lambda_\Psi)\otimes_{\Fq}H$ by
\eqref{eq:split-different}. This multiplication is injective:
$\Lambda_\Psi$ is a torsion-free $\A$-module, so its constant
extension embeds in $H(t)\otimes_K(K\otimes_{\A}\Lambda_\Psi)$,
and $\partial_xp\ne0$ acts invertibly on the latter space. Since
the constant extension of $\delta_p$ is an isomorphism, there is a unique
$\varpi_{\Psi,\omega}\in H\otimes_{\Fq}\Lambda_\Psi$ satisfying
\begin{equation}\label{eq:delta-omega}
 (\operatorname{id}_H\otimes\delta_p)\bigl((\partial_xp)\varpi_{\Psi,\omega}\bigr)
 =1\otimes\omega_\Phi.
\end{equation}
Define
\begin{equation}\label{eq:additive-character-omega}
 \psi_\omega(\bar a)
 =(\operatorname{id}_H\otimes\exp_\Psi)\left(
 (\operatorname{id}_H\otimes a(\theta)/p(\theta))
 \varpi_{\Psi,\omega}\right)
 \in H\otimes_{\Fq}\Psi[\pideal].
\end{equation}

\begin{lemma}[The scalar-extended additive character]\label{lem:additive-character}
The map \eqref{eq:additive-character-omega} is well-defined and
$\A$-linear on $\A/\pideal$. For each period in a
tensor presentation of $\varpi_{\Psi,\omega}$ it recovers the
additive character of \cite[Definition~4.9]{GazdaMaurischat2021},
with $u_{\pideal}=p$ and $z_{\pideal}=p^{-1}$.
\end{lemma}

\begin{proof}
Write $\varpi_{\Psi,\omega}=\sum_{\ell=1}^{s}c_\ell\otimes\pi_\ell$, with
$c_\ell\in H$ and $\pi_\ell\in\Lambda_\Psi$. Then
\[
 \psi_\omega(\bar a)
 =\sum_\ell c_\ell\otimes
 \exp_\Psi\!\left(\frac{a(\theta)\pi_\ell}{p(\theta)}\right).
\]
Replacing $a$ by $a+pb$ adds $b(\theta)\pi_\ell\in\Lambda_\Psi$
to each exponential argument and hence changes no value. For
$c\in\A$, the exponential functional equation gives
\[
 \psi_\omega(\overline{ca})
 =(\operatorname{id}_H\otimes\Psi_c)\psi_\omega(\bar a).
\]
Applying $\operatorname{id}_H\otimes\Psi_p$ gives zero, since each
resulting exponential is $\exp_\Psi(a(\theta)\pi_\ell)=0$.
This proves that the value lies in $H\otimes_{\Fq}\Psi[\pideal]$
and that the character is well defined and $\A$-linear.
For each period $\pi_\ell$ this is the character of
\cite[Definition~4.9]{GazdaMaurischat2021}, since $p$ and $p^{-1}$
are the elements $u_{\pideal}$ and $z_{\pideal}$ of that definition.
Linearity in the $H$-factor gives the stated comparison.
\end{proof}

\subsection{Evaluation of Gauss--Thakur sums}
\label{subsec:GT-evaluation}

\begin{definition}[Gauss--Thakur sum {\cite[Definition~4.1]{GazdaMaurischat2021}}]
\label{def:Gauss-Thakur-revised}
For a character $\chi:(\A/\pideal)^\times\to\mathbb F_{q^n}^\times$
and a presentation
$\varpi_{\Psi,\omega}=\sum_{\ell=1}^{s}c_\ell\otimes\pi_\ell$,
define its scalar Gauss--Thakur sum by
\begin{equation*}
 g_{\Psi,\omega}(\chi)
 =-\sum_{\ell=1}^{s}\sum_{\bar b\in(\A/\pideal)^\times}
 c_\ell\chi(\bar b)^{-1}
 \exp_\Psi\!\left(\frac{b(\theta)\pi_\ell}{p(\theta)}\right).
\end{equation*}
This is independent of the tensor presentation. It is the image of
$-\sum_{\bar b}\chi(\bar b)^{-1}\otimes\psi_\omega(\bar b)$ under
the multiplication map
\[
 \mathbb F_{q^n}\otimes_{\Fq}H\otimes_{\Fq}\Cinf\longrightarrow\Cinf,
 \qquad b\otimes c\otimes z\longmapsto bcz.
\]
The tensor sum uses the convention of
\cite[Definition~4.1]{GazdaMaurischat2021}.

\end{definition}

\begin{proposition}[Gauss--Thakur sums at the chosen branch]
\label{prop:GM-specialized}
Use the formal period of \eqref{eq:delta-omega}, the additive character $\psi_\omega$ of \eqref{eq:additive-character-omega}, and the Gauss--Thakur convention of Definition~\ref{def:Gauss-Thakur-revised}. Then, for $0\le j<d$,
\begin{equation}\label{eq:GM-specialized}
 \omega_\Phi(\lambda_{[j]})
 =\partial_xp(\lambda_{[j]})\,g_{\Psi,\omega}(\vartheta_{\pideal,\lambda}^{q^{Nj}}).
\end{equation}
Along the $q^N$-orbit of $\lambda$, the formula becomes
\begin{align*}
 g_{\Psi,\omega}\!\left(
 \vartheta_{\pideal,\lambda}^{q^{Nj}}\right)
 &=\bigl(\partial_xp(\lambda)\bigr)^{-q^{Nj}}
 \omega_\Phi\!\left(\lambda^{q^{Nj}}\right)
 \\
 &=\bigl(\partial_xp(\lambda)\bigr)^{-q^{Nj}}
 H_\Phi\!\left(\wt\pi_\Phi;\lambda^{q^{Nj}}\right).
\end{align*}
In terms of the $t$-derivative, the same value is
\begin{equation*}
 g_{\Psi,\omega}\!\left(
 \vartheta_{\pideal,\lambda}^{q^{Nj}}\right)
 =-\bigl((\lambda-\zeta)^2\frac{dp}{dt}(\lambda)\bigr)^{-q^{Nj}}
 \omega_\Phi\!\left(\lambda^{q^{Nj}}\right).
\end{equation*}
\end{proposition}

\begin{proof}
We first check which character occurs in the scalar action.
For
$\bar a\in(\A/\pideal)^\times$, replace $\bar b$ by
$\bar c=\bar a\bar b$ in the defining finite sum. For
this the sum is kept in the formal tensor product
$\mathbb F_{q^n}\otimes_{\Fq}H\otimes_{\Fq}\Cinf$, and the
multiplication map $b\otimes c\otimes z\mapsto bcz$ is applied only
at the final step. Additivity and
$\A$-linearity of $\psi_\omega$ give
\begin{align*}
 &(\operatorname{id}_{\mathbb F_{q^n}}\otimes
 \operatorname{id}_H\otimes\Psi_a)
 \Bigl(-\sum_{\bar b\in(\A/\pideal)^\times}
 \chi(\bar b)^{-1}\otimes\psi_\omega(\bar b)\Bigr)\\
 &\quad=-\sum_{\bar c\in(\A/\pideal)^\times}
 \chi(\bar a^{-1}\bar c)^{-1}\otimes\psi_\omega(\bar c)\\
 &\quad=(\chi(\bar a)\otimes1\otimes1)
 \Bigl(-\sum_{\bar c\in(\A/\pideal)^\times}
 \chi(\bar c)^{-1}\otimes\psi_\omega(\bar c)\Bigr).
\end{align*}
For $a\in\pideal$ both sides are zero, with
$\chi(0)=0$. Thus the
$\A$-action on the Gauss eigenspace is evaluation by $\chi$, not by
its inverse. The minus sign and inverse character inside the sum are
exactly those of \cite[Definition~4.1]{GazdaMaurischat2021}.

By Lemma~\ref{lem:Tate-algebra-comparison}, the analytic algebra and
evaluation maps in this paper agree with those used in
\cite[Theorem~5.2]{GazdaMaurischat2021}. That theorem gives
\begin{equation*}
 \begin{split}
 (\chi\widehat\otimes\operatorname{id})(\delta_p(\eta))
 &=-\sum_{\bar b\in(\A/\pideal)^\times}
 \chi(\bar b)^{-1}\otimes
 \exp_\Psi\!\left(\frac{b(\theta)\eta}{p(\theta)}\right),\\
 &\hspace{30mm}\eta\in\dfrak_{\A/\Fq[p]}\Lambda_\Psi.
 \end{split}
\end{equation*}
Its period map is the map of
\cite[Definition~4.9]{GazdaMaurischat2021}, and
Lemma~\ref{lem:additive-character} identifies the additive character
with the one used here. Extend this identity formally by $H$ and
apply it to $\eta=(\partial_xp)\varpi_{\Psi,\omega}$.
If $\partial_xp=\sum_\ell a_\ell\otimes c_\ell$, the equivariance
just checked shows that this coefficient acts on the
Gauss sum as
\[
 \sum_\ell a_\ell(\lambda_{[j]})\otimes c_\ell\otimes1.
\]
By \eqref{eq:delta-omega}, the value of $\delta_p(\eta)$
is $1\otimes\omega_\Phi$. Applying the multiplication map
$b\otimes c\otimes v\mapsto bcv$ to both sides of the theorem
therefore gives $\omega_\Phi(\lambda_{[j]})$ on the left and
$\partial_xp(\lambda_{[j]})g_{\Psi,\omega}(\vartheta_{\pideal,\lambda}^{q^{Nj}})$
on the right.
This proves \eqref{eq:GM-specialized}.

Finally, $\partial_xp$ has coefficients in $H$, which is fixed by
$q^N$-Frobenius. Hence
\[
 \partial_xp(\lambda_{[j]})
 =\bigl(\partial_xp(\lambda)\bigr)^{q^{Nj}}\ne0
\]
by Lemma~\ref{lem:split-different}. Division gives the first formula.
Equation~\eqref{eq:omega-H-equality} gives the expression using
$H_\Phi$, and the chain rule~\eqref{eq:px-pt} gives the final formula.
The chain rule contributes the displayed minus sign.
\end{proof}

\begin{remark}\label{rem:Gauss-normalization}
The additive character $\psi_\omega$ takes values in
$H\otimes_{\Fq}\Psi[\pideal]$, and $g_{\Psi,\omega}(\chi)$ is
the $H$-linear combination of Gauss--Thakur sums formed from the periods
in a tensor presentation of $\varpi_{\Psi,\omega}$. The formulas of
Proposition~\ref{prop:GM-specialized} hold for this specific
character.
\end{remark}

\section*{Acknowledgements}
Chuangqiang Hu was supported by the National Natural Science Foundation of China (No.~12441107). Xiao-Min Huang was supported by the National Natural Science Foundation of China (No.~12501094). The authors acknowledge the use of ChatGPT 6 to assist with language editing and manuscript revision.

\end{document}